\documentclass[reqno]{amsart}
\usepackage{amsmath, amssymb, amsthm, mathtools, amsfonts, mathrsfs, bm}
\usepackage[utf8]{inputenc}
\usepackage[all]{xy}
\usepackage{graphicx, epsfig, verbatim, multirow, diagbox, tabularx, array, makecell}
\usepackage[table]{xcolor}
\usepackage{enumerate, color, enumitem}
\usepackage{subcaption}
\usepackage[colorlinks,allcolors=blue]{hyperref} 
\usepackage{cleveref}
\usepackage{appendix,pdfsync}
\usepackage{float}
\usepackage{setspace}
\usepackage{xpatch}
\usepackage{pst-node}
\usepackage{tikz-cd}
\usetikzlibrary{positioning}
\usetikzlibrary{shapes,arrows}
\usepackage{tikz}

\newtheorem{theorem}{Theorem}[section]
\newtheorem{proposition}[theorem]{Proposition}
\newtheorem{corollary}[theorem]{Corollary}
\newtheorem{lemma}[theorem]{Lemma}
\newtheorem{remark}[theorem]{Remark}

\newtheorem{conjecture}[theorem]{Conjecture}

\theoremstyle{definition}
\newtheorem{definition}[theorem]{Definition}
\newtheorem{example}[theorem]{Example}
\newtheorem{question}[theorem]{Question}

\newcommand{\mcl}[1]{\mathcal{#1}}
\def\c{C_n^p}
\def\g{\mcl{G}_{mn}}
\def\gp{\mcl{G}_{mn}'}
\def\s{\sigma}
\def\t{\tau}
\def\iff{if and only if }
\newcommand{\D}[2]{\Delta_{#1}(#2)}
\newcommand{\Dt}[2]{\Delta_{#1}^t(#2)}
\newcommand{\un}[2]{#1\cup\{#2\}}
\newcommand{\di}[2]{#1\setminus\{#2\}}
\newcommand{\bb}[1]{\mathbb{#1}}
\newcommand{\m}[1]{m_{#1}}
\newcommand{\M}{\mathcal{M}}
\newcommand{\C}{\bm{\mathcal{C}}}
\newcommand{\X}{\mathcal{X}}

\newcommand{\bt}{\beta}

\newcommand{\homo}[2]{${#1}:\mathbb{Z}^{#2}$}

\newcommand\cfill[1]{\cellcolor{cyan!30}{#1}}

\begin{document}
	\title{Total $2$-cut complexes of powers of cycle graphs and Cartesian products of certain graphs}
	
	\author{Pratiksha Chauhan}
	\address{School of Mathematical and Statistical Sciences, IIT Mandi, India}
	\email{d22037@students.iitmandi.ac.in}
	
	\author{Samir Shukla}
	\address{School of Mathematical and Statistical Sciences, IIT Mandi, India}
	\email{samir@iitmandi.ac.in}
	
	\author{Kumar Vinayak}
	\address{Department of Mathematics, University of Kentucky, USA}
	\email{kumar.vinayak@uky.edu}
	
	
	\begin{abstract}
		For a positive integer $k$, the {\it total $k$-cut complex} of a graph $G$, denoted as $\Dt{k}{G}$, is the simplicial complex whose facets are $\s \subseteq V(G)$ such that $|\s| = |V(G)|-k$ and the induced subgraph $G[V(G) \setminus \s]$ does not contain any edge.  
		These complexes were introduced by Bayer et al.\ in \cite{Bayer2024TotalCutcomplex} in connection with commutative algebra. In the same paper, they studied the homotopy types of these complexes for various families of graphs, including cycle graphs $C_n$, squared cycle graphs $C_n^2$, and  Cartesian products of complete graphs and path graphs $K_m \square P_2$ and $K_2 \square P_n$. 
		
		In this article, we extend the work of Bayer et al.\ for these families of graphs. 
		We focus on the complexes $\Dt{2}{G}$ and determine the homotopy types of these complexes for three classes of graphs: (i) $p$-th powers of cycle graphs $C_n^p$ (ii) $K_m \square P_n$ and (iii) $K_m \square C_n$. 
		Using discrete Morse theory, we show that these complexes are homotopy equivalent to wedges of spheres. We also give the number and dimension of spheres appearing in the homotopy type.
		
		Our result on powers of cycle graphs $C_n^p$ proves a conjecture of Shen et al.\ about the homotopy type of the complexes $\Delta_2^t(C_n^p)$.
		
	\end{abstract}
	
	\maketitle
	
	\section{Introduction}
	All graphs are assumed to be finite and simple, meaning that they have no loops or multiple edges. The vertex set and edge set of a graph $G$ are denoted by $V(G)$ and $E(G)$, respectively.
	
	Graph complexes are simplicial complexes constructed from graphs by encoding their combinatorial structure into higher-dimensional simplices. In topological combinatorics, graph complexes are one of the main objects of study.
	Various topological invariants of these complexes, such as homotopy type, connectivity, homology, and their refinements, provide powerful tools for understanding combinatorial properties of graphs.
	A classical and one of the most celebrated examples of graph complexes is the neighbourhood complex, introduced by L.\ Lov{\'a}sz in 1978 \cite{Lovasz1978}. Lov{\'a}sz gave a lower bound for the chromatic number of a graph in terms of the topological connectivity of its neighbourhood complex, and using this bound, he proved the Kneser conjecture \cite{Kneser}.  
	
	Graph complexes are prevalent across the literature, and have applications in various areas of mathematics, engineering, and data analysis.  They are central to graph theory and combinatorics,  where the topology of these complexes is related to the combinatorics of the underlying graphs (see \cite{BabsonKozlov2006, SamirAnuragPriyavrat2022, Matsushita2023, Meshulam2001, Meshulam2003, ShuklaSanthanam2023}),   and have found many uses in data analysis and machine learning through simplicial representation of networks and topological data analysis (see \cite{Carlsson2009, EdelsbrunnerHarer2010, GhristBassett2015}). 
	Graph complexes also appear extensively in  commutative algebra, where they are intricately related to monomial ideals and resolutions (see \cite{DochtermannEngstrom2009, DochtermannEngstrom2012, HerzogHibi2011, MillerSturmfels2005, ReinerRoberts2000, Stanley1996, Woodroofe2014}). 
	In representation theory, they are used in the study of symmetric group actions on simplicial complexes and associated homological representations
	(see \cite{ Bate_MatchingComplex2023, BjornerLovaszZivaljevic1994, BoucMatchingComplex}). 
	Graph complexes have connections to geometric group theory, where they serve as combinatorial models for classifying spaces and group actions (see \cite{Gromov87, ShuklaGuptaSarkar2025, Zaremsky2022}), and also have applications in random topology (see \cite{Bobrowaski2018, Kahle2014, MeshulamWallach2009, ShuklaYogeshwaran2020}). For comprehensive treatments of graph complexes, we refer the reader to the books \cite{JonssonBook},  \cite{Kozlov2008} and \cite{Matousek2003}. 
	
	Recently, Bayer et al.\ in \cite{Bayer2024TotalCutcomplex} and \cite{Bayer2024Cutcomplex} introduced two new families of graph complexes: total cut complexes and cut complexes.  
	For $k \geq 1$, the {\it total $k$-cut complex} of a graph $G$, denoted as $\Dt{k}{G}$, is the simplicial complex whose facets (maximal simplices) are $\s \subseteq V(G)$ such that $|\s| = |V(G)|-k$ (where $| \cdot |$ denotes cardinality) and the induced subgraph $G[V(G) \setminus \s]$ does not contain any edge (see Example \ref{example}). 
	The {\it $k$-cut complex} of a graph $G$, denoted as $\D{k}{G}$, is the simplicial complex whose facets are $\s \subseteq V(G)$ such that $|\s| = |V(G)|-k$ and the induced subgraph $G[V(G) \setminus \s]$ is disconnected. 
	One of the main motivations behind these complexes is a famous theorem of Ralf Fr{\"o}berg connecting commutative algebra and graph theory through topology (see \cite[Theorem 1]{Froberg1990}, \cite[p.\ 274]{Eagon1998}). 
	
	It can be easily observed that $\Dt{2}{G} = \D{2}{G}$. In fact, it is mentioned in \cite{Bayer2024TotalCutcomplex} that both the complexes $\Delta_k^t(G)$ and $\D{k}{G}$ can be considered as a generalization of the complex $\Dt{2}{G}$.
	
	\begin{example}\label{example}
		Let $G$ be the graph given in Figure \ref{Figure1}. Then the set of maximal simplices of $\Dt{2}{G}$ is $\{ \{1, 2, 3\}, \{1, 3, 5\},\{2, 3, 4\}, \{2, 4, 5\} \}$ (see Figure \ref{Figure2}).
	\end{example}
	\begin{figure}[h!]
		\begin{subfigure}[]{0.45\textwidth}
			\centering
			\begin{tikzpicture}[scale=.35,auto=left,every node/.style={circle,fill=blue!40, inner sep=1.5pt, font=\tiny}]
				\node (n1) at (1,1) {1};
				\node (n2) at (4.5, 1) {2};
				\node (n3) at (4.5, 4.5) {3};
				\node (n4) at (1, 4.5) {4};
				\node (n5) at (6.75, 2.75) {5};
				
				\foreach \from/\to in {n1/n2,n2/n3,n3/n4, n4/n1, n2/n5, n5/n3}
				\draw (\from) -- (\to);
			\end{tikzpicture}
			\caption{G}\label{Figure1}
		\end{subfigure}
		\begin{subfigure}[]{0.45\textwidth}
			\centering
			\begin{tikzpicture}[scale=0.25, vertices/.style={draw, fill=black, circle, inner sep=0.5pt}]
				\node[vertices, label=left:{\small{3}}] (a) at (0, 0) {};
				\node[vertices, label=right:{\small{1}}] (b) at (3.5, 0.2) {};
				\node[vertices, label=above:{\small{2}}] (c) at (5, -2) {};
				\node[vertices, label=right:{\small{5}}] (d) at (8,4) {};
				\node[vertices, label=right:{\small{4}}] (e) at (6, -5) {};
				
				\foreach \to/\from in {a/b,a/c,b/c, b/d, c/d, a/d, a/e, c/e, d/e}
				\draw [-] (\to)--(\from);
				
				\filldraw[fill=blue!30, draw=black] (0, 0)--(3.5, 0.2)--(5, -2)--cycle;
				\filldraw[fill=blue!30, draw=black] (0, 0)--(3.5, 0.2)--(8, 4)--cycle;
				\filldraw[fill=blue!30, draw=black] (0, 0)--(5, -2)--(6, -5)--cycle;
				\filldraw[fill=blue!30, draw=black] (5, -2)--(6, -5)--(8, 4)--cycle;
			\end{tikzpicture}
			\caption{$\Dt{2}{G}$}\label{Figure2}
		\end{subfigure}
		\caption{}
	\end{figure}
	
	In this article, we consider the total $2$-cut complexes (hence $2$-cut complexes) of three families of graphs: (i) powers of cycle graphs, (ii) Cartesian product of complete graphs and path graphs, and (iii) Cartesian product of complete graphs and cycle graphs.
	
	For a graph $G$ and a positive integer $p$, the $p$-th power graph of $G$ is a graph $G^p$ with $V(G^p)=V(G)$ and $\{u, v\}\in E(G^p)$ if and only if there exists a path between $u$ and $v$ of length at most $p$ in $G$ (here, the length of a path is the number of edges in the path). Clearly, $G^1 = G$.
	
	Let $\Gamma$ be a group and let $S\subset\Gamma$ be a subset not containing the identity element.
	The \textit{Cayley graph}  of $\Gamma$ with respect to $S$ is the graph $G(\Gamma,S)$ having vertex set $\Gamma$, and   $\{u, v\}\in E(G(\Gamma,S))$ if and only if $uv^{-1} \in S \cup S^{-1}$, where $S^{-1} = \{x^{-1} \ | \ x \in S\}$.
	For $n\geq 2$ and $S \subset \{ 1, 2, \ldots, n-1\}$, the \textit{circulant graph} $C_n(S)$ is a Cayley graph of $\mathbb{Z}_n$, the cyclic group of order $n$. 
	
	For $n\geq 3$, let $C_n$ denote the cycle graph on $n$ vertices $\{0,1,\ldots,n-1\}$. Observe that for $p \geq 1$, $V(\c) = V(C_n) = \{0,1,\ldots,n-1\}$ and $E(\c) =\{ \{i,i+j\ (\text{mod $n$})\}\,|\,0\le i\le n-1 $ and $1\le j\le p$\}. 
	It is easy to check that for $n\ge 2p+1$, $\c$ is a circulant graph $C_n(\{1, 2,\ldots, p\})$.
	
	In \cite{Bayer2024TotalCutcomplex}, the authors studied $\Dt{k}{C_n}$ (Theorem 4.15) and $\Dt{2}{C_n^2}$ (Proposition 4.19). They proved that these complexes (if nonvoid) are homotopy equivalent to wedges of spheres. 
	In \cite{shen2025homotopy}, the authors determined the homotopy type of $\Dt{3}{C_n^2}$ (Theorem 1.5) and, for $k\ge 3$, the homotopy type of $\Delta_k^t(C_{n}^2)$ when $n = 3k+1,\,3k+2$ (Theorem 1.6). 
	
	The $k$-cut complexes of powers of cycle graphs have been studied by many authors. In \cite[Proposition 7.11]{Bayer2024Cutcomplex}, the authors proved that for $k \geq 3$, $\D{k}{C_n}$ is shellable\footnote{A simplicial complex $\Delta$ is called \textit{shellable} if its facets can be arranged in a linear order $F_1, F_2, \ldots, F_t$ in such a way that for all $2\leq j\leq t$, the subcomplex $(\bigcup_{i=1}^{j-1} F_i) \cap F_j$ is pure and of dimension $|F_j|-2$.}, and in \cite[Theorem 1.4]{Shellability2025}, it has been proved that $\D{3}{C_n^2}$ is shellable for $n\geq k+6$. 
	In \cite{ShuklaChauhan2025}, the authors proved that $\D{3}{\c}$ is shellable for $n\geq 6p-3$. Hence all of these complexes are wedges of spheres up to homotopy type. 
	
	For general $k \geq 2$ and $p \geq 2$, the topology of the complexes $\Dt{k}{\c}$ and $\D{k}{\c}$ is not known. For $k = 2$, Shen et al.\ made the following conjecture about the homotopy type of the complexes $\Dt{k}{\c}$. 
	
	\begin{conjecture}[{\cite[Conjecture 5.1]{shen2025homotopy}}]\label{conjecture:shen}
		For $p \ge 3$ and $n\ge 3p + 1$, $\Dt{2}{\c}\simeq\mathbb{S}^{n - 4}$.
	\end{conjecture} 
	
	One of the main results of this article is that Conjecture \ref{conjecture:shen} is true. It is easy to check that for $n < 2p+2$, the complex $\Dt{2}{\c}$ is void. For $n\geq 2p+2$, we prove the following.
	
	\begin{theorem}[Theorem \ref{theorem:powers_of_cycle}]\label{theorem:Intro_powers_of_cycle} 
		Let $n\geq 2p+2$. Then
		$$\Dt{2}{\c}\simeq 
		\begin{cases}
			\mathbb{S}^{\frac{n-4}{2}} & \text{ if }n=2p+2,\\
			\mathbb{S}^{n-4} & \text{ if }n\ge 3p+1.
		\end{cases}$$
	\end{theorem}
	
	For $2p+3 \leq n\le3p$, using SageMath, we computed the homology groups of the complex $\Dt{2}{\c}$ for certain values of $n$ and $p$. Based on our calculation, we conjecture about the homotopy type of $\Dt{2}{\c}$ for all $2p+3 \leq n\le3p $ (see Conjecture \ref{conjecture:2-cut power cycle} of Section \ref{section:future_direction})
	
	The \textit{Cartesian product} of graphs $G$ and $H$, denoted by $G \Box H$, is the graph where $V(G \Box H) = V(G) \times V(H)$ and $\{(u_1,v_1), (u_2,v_2)\} \in E(G \Box H)$ if and only if either $u_1=u_2$ and $\{v_1,v_2\} \in E(H)$, or $v_1=v_2$ and $\{u_1,u_2\}\in E(G)$.
	
	For a positive integer $n$, let $K_n$ and $P_n$ denote the complete graph and the path graph on $n$ vertices, respectively. In \cite{Bayer2024TotalCutcomplex}, it is proved that $\Dt{2}{K_m \square P_2}$ (Theorem 4.8) and $\Dt{2}{K_2 \square P_n}$ (Theorem 4.16) are homotopy equivalent to wedges of spheres. 
	Further, in \cite[Theorem 3.2]{chandrakar2024topology}, the authors proved that for $2\le k\le n$, $\Dt{k}{K_2\square P_n}$ is homotopy equivalent to a wedge of spheres. In this article, we extend the results of \cite{Bayer2024TotalCutcomplex} to all $K_m$ and all $P_n$, and prove the following.
	
	\begin{theorem}[Theorem \ref{theorem:K_m[]P_n}]\label{theorem:Intro_Cartesian_path}
		For $m,n\geq 2,$ $\Dt{2}{K_m \square P_n} \simeq \bigvee_{(m-1)(n-1)} \mathbb{S}^{mn-4}.$
	\end{theorem}
	
	We also investigate the total $2$-cut complexes of $K_m \square C_n$ and  determine their homotopy types. 
	
	\begin{theorem} [Theorem \ref{theorem:K_m[]C_n}]\label{theorem:Intro_Cartesian_cycle}
		For $m \geq 2$ and $ n\geq 4,$ $\Dt{2}{K_m \square C_n} \simeq \bigvee_{n(m-1)+1} \mathbb{S}^{mn-4}.$
	\end{theorem}
	
	Beyond the total $2$-cut complexes, we investigate the total $3$-cut and $3$-cut complexes of $K_m \square P_n$ and $K_m \square C_n$. We computed their homology groups for small values of $m$ and $n$ using SageMath. These computations lead us to propose conjectures about their homotopy types (see Conjectures \ref{conjecture:total 3-cut K_m box P_n} to \ref{conjecture 3-cut K_m box C_n}).
	
	This paper is organized as follows: In Section \ref{section:preliminaries}, we recall the necessary preliminaries related to graph theory, simplicial complexes and discrete Morse theory. 
	Section \ref{section:basic_results} is divided into three subsections (Sections \ref{subsection:powers_of_cycle_graphs} to \ref{subsection:Km_Cn}) where we prove  Theorems \ref{theorem:Intro_powers_of_cycle} to \ref{theorem:Intro_Cartesian_cycle}, respectively. 
	Finally, Section \ref{section:future_direction} proposes several conjectures and questions based on our SageMath computations.  
	
	\section{Preliminaries}\label{section:preliminaries}
	This section presents some basic definitions and results used in this article.
	
	\subsection{Graph}
	A \textit{graph} $G$ is a pair $(V(G), E(G))$, where $V(G)$ is its vertex set and $E(G) \subseteq \binom{V(G)}{2}$ is the edge set. For any $u,v\in V(G)$, we say that $u$ and $v$ are adjacent if $\{u,v\} \in E(G)$. 
	We write $u\sim v$ for adjacency and $u\nsim v$ for non-adjacency. A {\it subgraph} $H$ of $G$ is a graph with $V(H) \subseteq V(G)$ and $E(H) \subseteq E(G)$.
	For a subset $U \subseteq V(G)$, the \textit{induced subgraph} $G[U]$ is the subgraph with $V(G[U]) = U$ and $E(G[U]) = \{\{a, b\} \in E(G) \ | \ a, b \in U\}$.
	
	For $u,v\in V(G)$, a {\it path} from $u$ to $v$ is a sequence of distinct vertices $u=v_0,v_1, \ldots, v_n=v$ such that $v_i \sim v_{i+1}$ for all $0 \leq i\leq n-1$. The {\it length} of a path is the number of edges in the path. 
	A graph is \textit{connected} if there exists a path between each pair of its vertices; otherwise, it is \textit{disconnected}.
	
	The \textit{complete graph} $K_n$ on $n$ vertices is the graph in which any two distinct vertices are adjacent. For $n\ge 1$, the \textit{path graph} $P_n$ is the graph with $V (P_n) = \{0,1,\ldots , n-1\}$ and $E(P_n) = \{\{i, i+1\}\,|\,0\le i \le n-2\}$. 
	For $n\ge 3$, the \textit{cycle graph} $C_n$ is the graph with $V (C_n) = \{0,1,\ldots , n-1\}$ and $E(C_n) = \{\{i, i+1\}\,|\,0\le i \le n-2\} \cup\{\{0, n-1\}\}$.
	
	We refer the reader to \cite{bondy1976graph} and \cite{west} for more details about the graph terminologies used in this article.
	
	\subsection{Simplicial complex}
	A {\it finite abstract simplicial complex} $\Delta$ is a collection of finite sets such that if $\t \in\Delta$ and $\s \subset \t$, then $\s \in\Delta$. 
	The elements of $\Delta$ are called {\it simplices} of $\Delta$. If $\s \subset \t$, we say that $\s$ is a {\it face} of $\t$. The \textit{dimension of a simplex} $\s$ is equal to $|\s| - 1$. The \textit{dimension of an abstract simplicial complex} is the maximum of the dimensions of its simplices. 
	If a simplex has dimension $d$, it is said to be $d$-{\it dimensional}. The $0$-dimensional simplices are called \textit{vertices} of $\Delta$. An abstract simplicial complex which is an empty collection of sets is called the \textit{void} abstract simplicial complex, and is denoted by $\emptyset$.
	A simplex that is not a face of any other simplex is called a {\it maximal simplex} or \textit{facet}. A \textit{subcomplex} $\Delta'$ of $\Delta$ is a simplicial complex such that $\s\in\Delta'$ implies $\s\in\Delta$.
	
	In this article, we consider any simplicial complex as a topological space, namely, its geometric realization (see \cite{Kozlov2008} for details). For terminologies related to algebraic topology, we refer to \cite{hatcher2005algebraic}.
	
	\subsection{Discrete Morse theory}
	Discrete Morse theory was introduced by Robin Forman \cite{forman1998} in 1998 as a combinatorial adaptation of classical Morse theory. For more details, we refer the reader to \cite[Chapter 4]{JonssonBook} and \cite[Chapter 11]{Kozlov2008}, which serve as the primary sources for the definitions and results in this section.
	
	Throughout this section, let $X$ be a set and let $\Sigma$ be a finite family of finite subsets of $X$.
	\begin{definition}
		A \textit{matching} on $\Sigma$ is a family $\M$ of pairs $\{\s,\t\}$ with $\s,\t\in\Sigma$ such that no set is contained in more than one pair in $\M$.
	\end{definition}  
	A set $\s\in\Sigma$ is said to be \textit{matched} in $\M$ if it is contained in some pair of $\M$; for convenience, we write $\s\in\M$. Otherwise, $\s$ is said to be \textit{critical} or \textit{unmatched} with respect to $\M$, and we write $\s\notin\M$. 
	
	\begin{definition}
		A matching $\M$ is said to be \textit{partial} if $\{\s,\t\}\in\M$ implies either $\s\prec\t$ (\textit{i.e.}, $\s\subset\t$ and no $\rho\in\Sigma$ satisfies $ \s\subset\rho\subset\t$), or $\s\succ \t$.
	\end{definition}
	Note that $\M$ is a partial matching on $\Sigma$ if and only if there exists $\Sigma' \subset\Sigma$ and an injective map $\mu: \Sigma'\rightarrow \Sigma\setminus \Sigma'$ such that $\mu(\s)\succ \s$ for all $\s\in\Sigma'$. 
	
	\begin{definition}
		A partial matching $\M$ on $\Sigma$ is called {\it acyclic} if there does not exist a cycle
		$$\mu(\s_1) \succ \s_1 \prec \mu(\s_2) \succ \s_2 \prec \mu(\s_3) \succ \s_3 \prec \dots \prec \mu(\s_t) \succ \s_t \prec \mu(\s_1),\ t\ge 2.$$
	\end{definition}
	
	\begin{definition}
		Let $x\in X$. A matching $\M_x$ is called an \textit{element matching} on $\Sigma$ using $x$ if every pair in $\M_x$ is of the form $\{\s\setminus\{x\},\s\cup\{x\}\}$ for some $\s\in\Sigma$.
	\end{definition}  
	It is easy to observe that any element matching is a partial matching. 
	
	\begin{lemma}[{\cite[Lemma 4.1]{JonssonBook}}]\label{lemma:Jonsson acyclic}
		Let $x\in X$. Define $$\M_x:= \{\{\s\setminus\{x\},\s\cup\{x\}\}\, |\, \s\setminus\{x\},\s\cup\{x\}\in\Sigma\}.$$ 
		Let $\M'$ be an acyclic matching on $\Sigma':= \{\s\in\Sigma\, |\, \s\notin\M_x\}$. Then $\M:= \M_x \cup\M'$ is an acyclic matching on $\Sigma$.
	\end{lemma}
	
	Let $\{x_0,x_1,\ldots,x_{n-1}\}$ be a subset of $X$. We define a sequence of element matchings $\M_{x_0},\M_{x_1},\ldots,$ $\M_{x_{n-1}}$ on $\Sigma$ as follows. 
	
	Let $\C_0:=\Sigma$. For each $0\le i\le n-1$, define \begin{equation}\label{equation:sequence_of_element_matchings}
		\begin{aligned}
			\M_{x_i} &:= \{\{\di{\s}{x_i},\un{\s}{x_i}\}\, |\, \di{\s}{x_i},\un{\s}{x_i}\in\C_i\};\\
			\C_{i+1} &:= \{\s\in\C_i\, |\, \s\notin\M_{x_i}\}. 
		\end{aligned}           
	\end{equation}
	
	Observe that $\C_{i+1} = \{\s\in\Sigma\, |\, \s\notin\M_{x_j}\text{ for all }0\le j\le i\}$. Moreover, the matchings $\M_{x_0},\M_{x_1},\ldots,$ $\M_{x_{n-1}}$ are pairwise disjoint.
	
	\begin{remark}\label{remark:C_i}
		Let $0\le i\le n-1$. Then the following results hold: 
		\begin{enumerate}[label=(\roman*)]
			\item \label{C_i subseteq} $\C_{i+1}\subseteq\C_i\subseteq\C_{i-1}\subseteq\ldots\subseteq\C_1\subseteq\C_0=\Sigma$.
			\item \label{in C_i to C_i+1} For any $\s\in\C_i$, $\s\in\C_{i+1}$ \iff $\di{\s}{i}\notin\C_i$ or $\un{\s}{i}\notin\C_i$.
			\item \label{in C_j+1} Let $\s\in\C_i$ and let $i\le j\le n-1$. If $\di{\s}{r}\notin\C_i$ or $\un{\s}{r}\notin\C_i$ for all $i\le r\le j$, then $\s\in\C_{j+1}$.
		\end{enumerate}     
	\end{remark}
	
	To define a matching on a simplicial complex, we consider it simply as the set of its simplices (or faces). 
	
	\begin{theorem}[{\cite[Theorem 4.14]{JonssonBook}}]\label{theorem:acyclic}
		Let $\Delta$ be a simplicial complex and let $\M$ be an acyclic matching on $\Delta$ such that the empty set is not critical. Then $\Delta$ is homotopy equivalent to a cell complex with one cell of dimension $d \ge 0$ for each critical face of $\Delta$ of dimension $d$ plus one additional $0$-cell.
	\end{theorem}
	
	Theorem \ref{theorem:acyclic} implies the following result.
	
	\begin{corollary}[{\cite[Theorem 4.8]{JonssonBook}}] \label{corollary:critical}
		If an acyclic matching on a simplicial complex $\Delta$ has all its critical faces in the same dimension $d$, then $\Delta$ is homotopy equivalent to a wedge of spheres of dimension $d$.
	\end{corollary}
	
	\section{Proofs}\label{section:basic_results} 
	In this section, we determine the homotopy type of $\Dt{2}{\c}$ (Theorem \ref{theorem:Intro_powers_of_cycle}), $\Dt{2}{K_m\square P_n}$ (Theorem \ref{theorem:Intro_Cartesian_path}), and $\Dt{2}{K_m\square C_n}$ (Theorem \ref{theorem:Intro_Cartesian_cycle}).
	
	Let $G$ be a graph, and let $u,v\in V(G)$. Recall that the notation $u\sim v$ denotes that the vertices $u$ and $v$ are adjacent in $G$, and $u\nsim v$ denotes that $u$ and $v$ are not adjacent.
	Let $A \subseteq V(G)$. For $B \subseteq A$, we say that $B$ is a disconnected $2$-set in $A$ if $B = \{u, v\}$ for some  $u \nsim v$.
	Throughout this article, $A^c$ denotes the complement of $A$ in $V(G)$.
	
	We begin by establishing some basic results that will be used later. The following remark is a direct implication of the definition of the total $2$-cut complex.
	
	\begin{remark}\label{remark:face of 2-cut complex}
		Let $G$ be a graph, and let $\s \subseteq V(G)$. Then $\s\in\Dt{2}{G}$ if and only if there exist $u, v\in\s^c$ such that $u \nsim v$, {\it i.e.}, $\{u, v\}$ is a disconnected $2$-set in $\s^c$. Equivalently, $\s \notin\Dt{2}{G}$ if and only if the induced subgraph $G[\s^c]$ is a complete graph. 
	\end{remark}
	
	\begin{proposition} \label{proposition:first matching}
		Let $G$ be a graph and $u,v\in V(G)$ with $u\neq v$. Define 
		\begin{align*}
			\M_v &:= \{\{\s\setminus\{v\},\s\cup\{v\}\}\, |\, \di{\s}{v}, \s\cup\{v\}\in\Dt{2}{G}\};\\
			\C&:=\{\s\in\Dt{2}{G}\, |\, \s\notin\M_v\}.
		\end{align*} 
		Then, for any $\s\in\Dt{2}{G}$, the following statements hold:
		\begin{enumerate}[label = (\roman*)]
			\item \label{sigma (not) in C iff}
			$\s\in\C$ if and only if $\s\cup\{v\}\notin\Dt{2}{G}$. Equivalently, $\s\in\C$ if and only if every disconnected $2$-set in $\s^c$ contains $v$.
			
			\item \label{sigma union u in C} 
			If $\s\in\C$ and $\un{\s}{u}\in\Dt{2}{G}$, then $\un{\s}{u}\in\C$.
			
			\item \label{sigma minus u in C} 
			If $\s\in\C$ and $u\sim w$ for all $w\in(\un{\s}{v})^c$, then $\di{\s}{u}\in\C$. 
		\end{enumerate}
	\end{proposition}
	
	\begin{proof}
		\begin{enumerate}[label=(\roman*)]
			\item Since $\Dt{2}{G}$ is a simplicial complex, $\un{\s}{v}\in\Dt{2}{G}$ implies $\di{\s}{v}\in\Dt{2}{G}$. It follows that $\M_v=\{\{\s\setminus\{v\},\s\cup\{v\}\}\, |\, \s\cup\{v\}\in\Dt{2}{G}\}.$
			By the definition of $\C$, $\s\in\C$ if and only if $\s\notin\M_v$, which means that $\{\s\setminus\{v\},\un{\s}{v}\}\notin\M_v$.
			Since $\{\s\setminus\{v\},\un{\s}{v}\}\notin\M_v$ if and only if $\s\cup\{v\}\notin\Dt{2}{G}$, it follows that $\s\in\C$ if and only if $\s\cup\{v\}\notin\Dt{2}{G}$. 
			
			By Remark \ref{remark:face of 2-cut complex}, $\s\in\Dt{2}{G}$ and $\s\cup\{v\}\notin\Dt{2}{G}$ if and only if every disconnected $2$-set in $\s^c$ contains $v$. Therefore, $\s\in\C$ if and only if every disconnected $2$-set in $\s^c$ contains $v$.
			
			\item Let $\s\in\C$ and $\un{\s}{u}\in\Dt{2}{G}$. If $u\in\s$, then clearly $\un{\s}{u}=\s\in\C$. Hence, assume that $u\notin\s$. By \ref{sigma (not) in C iff}, $\s\in\C$ implies $\un{\s}{v}\notin\Dt{2}{G}$. Hence $\un{(\un{\s}{u})}{v}\notin\Dt{2}{G}$. Therefore, $\un{\s}{u}\in\C$ by \ref{sigma (not) in C iff}.
			
			\item Let $\s\in\C$. Then $\un{\s}{v}\notin\Dt{2}{G}$ by \ref{sigma (not) in C iff}. By Remark \ref{remark:face of 2-cut complex}, $u_1\sim u_2$ for all $u_1,u_2\in(\un{\s}{v})^c$. 
			Now, assume further that $u\sim w$ for all $w\in(\un{\s}{v})^c$. Then $v_1\sim v_2$ for all $v_1,v_2\in(\un{(\di{\s}{u})}{v})^c$. By Remark \ref{remark:face of 2-cut complex}, $\un{(\di{\s}{u})}{v}\notin\Dt{2}{G}$. Therefore, $\di{\s}{u}\in\C$ by \ref{sigma (not) in C iff}.
		\end{enumerate}
	\end{proof} 
	
	To prove Theorems \ref{theorem:Intro_powers_of_cycle} to \ref{theorem:Intro_Cartesian_cycle}, we construct acyclic matchings on the total $2$-cut complex $\Dt{2}{G}$ of the graph $G$ under consideration. 
	This is done by selecting an appropriate subset of vertices and defining a sequence of element matchings as in Equation \eqref{equation:sequence_of_element_matchings}. 
	Our aim is to characterize the critical faces with respect to the matching. Then, using Theorem \ref{theorem:acyclic} and Corollary \ref{corollary:critical}, we determine the homotopy type of $\Dt{2}{G}$.
	
	\subsection{Powers of cycle graphs}\label{subsection:powers_of_cycle_graphs}
	In this section, we prove Theorem \ref{theorem:Intro_powers_of_cycle}. We assume throughout that $p\ge 1$ is an integer. Recall that the $p$-th power of a cycle graph $C_n^p$ is the graph on the vertex set $V(\c)=\{0,1,\ldots,n-1\}$, where for any distinct vertices $u,v\in V(\c)$, we have $u\sim v$ in $\c$ if and only if $v \equiv u \pm t \ (\text{mod $n$})$ for some $1\le t\le p$. 
	
	Before proving Theorem \ref{theorem:Intro_powers_of_cycle}, we establish several key results. We begin by defining a sequence of element matchings $\M_0,\M_1,\ldots,\M_{n-1}$ on $\Dt{2}{\c}$ and the sets $\C_0,\C_1,\ldots,\C_{n}$, using vertices $0,1,\ldots,n-1$, as in Equation \eqref{equation:sequence_of_element_matchings}.
	By Lemma \ref{lemma:Jonsson acyclic}, $\M:=\bigsqcup_{i=0}^{n-1} \M_i$ is an acyclic matching on $\Dt{2}{\c}$. To proceed, we characterize the sets $\C_{i+1}$ for $0\le i\le n-1$, thus identifying the critical faces with respect to the matching $\M$. 
	
	Note that $\C_0=\Dt{2}{\c}$ and $\C_1=\{\s\in\Dt{2}{\c}\, |\, \s\notin\M_0\}$. To characterize $\C_1$, we apply Proposition \ref{proposition:first matching} with $v=0$. We characterize $\C_{i+1}$ inductively in Proposition \ref{proposition:C_i and C_i+1 for 1<=i<=p} for all $1\leq i\leq p$. 
	For each $1\leq i\leq p$, we define a subset $\X_i\subseteq V(\c)$ as follows: $$\X_i:=\{i+p+1,i+p+2,\ldots,i+n-p-1\}.$$ 
	Clearly, for $n\ge2p+2$ and for all $1\leq i\leq p$, we have $\X_i\neq\emptyset$. 
	
	\begin{proposition}\label{proposition:C_i and C_i+1 for 1<=i<=p}
		Let $n\ge 2p+2$ and $1\le i\le p$. If $\s\in\C_i$, then $\s\in\C_{i+1}$ if and only if $\{1,2,\ldots,i\}\subseteq\s$ and $\X_i\not\subseteq\s$.
	\end{proposition}
	
	\begin{proof}   
		The proof proceeds by induction on $i$ for $1\le i\le p$. We first prove the base case $i=1$.
		Let $\s\in\C_1$. It suffices to show that $\s\in\C_2$ if and only if $1\in\s$ and $\X_1\not\subseteq\s$. 
		\begin{itemize}
			\item Suppose $\s\in\C_2$. We first show that $1\in\s$. Assume, for contradiction, that $1\notin\s$. Since $\s\in\C_1\subseteq\Dt{2}{\c}$, there exists a disconnected $2$-set in $\s^c$ by Remark \ref{remark:face of 2-cut complex}. 
			Moreover, by Proposition \ref{proposition:first matching} \ref{sigma (not) in C iff}, every disconnected $2$-set in $\s^c$ contains $0$. Hence, there exists $v\in V(\c)$ such that $v\nsim0$ and $\{0,v\}\subseteq\s^c$. Since $0\sim 1$, we get $v\neq 1$. 
			Therefore, $\{0,v\}\subseteq(\un{\s}{1})^c$, and by Remark \ref{remark:face of 2-cut complex}, $\un{\s}{1}\in\Dt{2}{\c}$. By Proposition \ref{proposition:first matching} \ref{sigma union u in C}, $\un{\s}{1}\in\C_1$. 
			Further, since $\di{\s}{1}=\s\in\C_1$, it follows by Remark \ref{remark:C_i} \ref{in C_i to C_i+1} that $\s\notin\C_2$, a contradiction. Thus, $1\in\s$.    
			
			Now, we show that $\X_1\not\subseteq\s$. Suppose $\X_1\subseteq\s$. Then $\X_1\subseteq\un{\s}{0}$. By the definition of $\c$, $1\nsim u$ if and only if $u\in\{p+2,p+3,\ldots,n-p\}=\X_1$. Hence $1\sim v$ for all $v\in (\un{\s}{0})^c$.
			By Proposition \ref{proposition:first matching} \ref{sigma minus u in C}, $\di{\s}{1}\in\C_1$. Since $\un{\s}{1}=\s\in\C_1$, we get $\s\notin\C_2$ by Remark \ref{remark:C_i} \ref{in C_i to C_i+1}, a contradiction. Therefore, $\X_1\not\subseteq\s$. 
			
			\item Conversely, assume that  $1\in\s$ and $\X_1\not\subseteq\s$. 
			Since $1\nsim u$ if and only if $u\in\X_1$, there exists $v\in\s^c$ such that $1\nsim v$. Moreover, since $1\sim0$, we have $v\neq0$. It follows that $\{1,v\}$ is a disconnected $2$-set in $((\di{\s}{1})\cup\{0\})^c$. 
			By Remark \ref{remark:face of 2-cut complex}, $\un{(\di{\s}{1})}{0}\in\Dt{2}{\c}$, and hence $\di{\s}{1}\notin\C_1$ by Proposition \ref{proposition:first matching} \ref{sigma (not) in C iff}. Since $\s\in\C_1$, we conclude that $\s\in\C_2$ by Remark \ref{remark:C_i} \ref{in C_i to C_i+1}. 
		\end{itemize}
		Therefore, the result holds for $i=1$.
		
		Let $1<i\le p$ and assume the result holds for all $1\le i'<i$; that is, if $\s\in\C_{i'}$, then $\s\in\C_{i'+1}$ if and only if $\{1,2,\ldots,i'\}\subseteq\s$ and $\X_{i'}\not\subseteq\s$.
		We now prove the result for $i$. Let $\s\in\C_i$. We show that $\s\in\C_{i+1}$ if and only if $\{1,2,\ldots,i\}\subseteq\s$ and $\X_i\not\subseteq\s$. 
		\begin{itemize}
			\item Suppose $\s\in\C_{i+1}$. By Remark \ref{remark:C_i} \ref{C_i subseteq}, $\C_{i+1} \subseteq \C_i\subseteq\ldots\subseteq\C_1$. Hence $\s\in\C_{j}$ and $\s\in\C_{j+1}$ for all $1\le j<i$. By the induction hypothesis, for each such $j$, we have $\{1,2,\ldots,j\}\subseteq\s$ and $\X_j\not\subseteq\s$.
			
			We first prove that $\{1,2,\ldots,i\}\subseteq\s$. Since $\{1,2,\ldots,j\}\subseteq\s$ for all $1\le j<i$, we have $\{1,2,\ldots,i-1\}\subseteq\s$.
			
			Suppose $i\notin\s$. Since $\s\in\C_1$, every disconnected $2$-set in $\s^c$ contains $0$ by Proposition \ref{proposition:first matching} \ref{sigma (not) in C iff}. 
			Moreover, $1< i\le p$ implies that $0\sim i$. It follows that there exists $v\in\di{\s^c}{i}$ such that $v\nsim0$ and $\{0,v\}\subseteq(\un{\s}{i})^c$. Hence $\un{\s}{i}\in\Dt{2}{\c}$ by Remark \ref{remark:face of 2-cut complex}. By Proposition \ref{proposition:first matching} \ref{sigma union u in C}, $\s\in\C_1$ and $i\neq 0$ implies that $\un{\s}{i}\in\C_1$.
			
			Since $\{1,2,\ldots,j\}\subseteq\s$ for all $1\le j<i$, we get $\{1,2,\ldots,j\}\subseteq\un{\s}{i}$ for all $1\le j<i$. Note that $n\ge2p+2$ implies $p+1\le n-p-1$. Further, since $1<i\le p$, for any $1\le j<i$, we have $1<i\le p<j+p+1\le j+n-p-1<i+n-p-1\le n-1$. 
			This gives $i\notin\X_{j}$ for all $1\le j<i$. Therefore, $\X_{j}\not\subseteq\s$ for all $1\le j<i$ implies $\X_{j}\not\subseteq\s\cup\{i\}$ for all $1\le j<i$.
			
			We have $\un{\s}{i}\in\C_1$, and for each $1\le j<i$, $\{1,2,\ldots,j\}\subseteq\un{\s}{i}$ and $\X_{j}\not\subseteq\s\cup\{i\}$. 
			Since $\un{\s}{i}$ satisfies the conditions of the induction hypothesis for each $1\le j<i$, by iterative application of the induction hypothesis, we have $\un{\s}{i}\in\C_{j+1}$ for all $1\le j<i$. 
			In particular, for $j=i-1$, we get $\un{\s}{i}\in\C_i$. Then $\di{\s}{i}=\s\in\C_i$ implies that $\s\notin\C_{i+1}$ by Remark \ref{remark:C_i} \ref{in C_i to C_i+1}, a contradiction. Hence $i\in\s$.
			
			Since $\{1,2,\ldots,i-1\}\subseteq\s$, we conclude that $\{1,2,\ldots,i\}\subseteq\s$.
			
			We now show that $\X_i\not\subseteq\s$. Suppose $\X_i\subseteq\s$. Since $1< i\le p$, we have $0\notin\X_i$. Moreover, by the definition of $\c$, $i\nsim u$ if and only if $u\in\X_i$. 
			It follows that $\X_i\subseteq\un{\s}{0}$, and $i\sim v$ for all $v\in (\un{\s}{0})^c$. By Proposition \ref{proposition:first matching} \ref{sigma minus u in C}, we get $\di{\s}{i}\in\C_1$. 
			
			For each $1\le j<i$, we have $\{1,2,\ldots,j\}\subseteq\s$ and $\X_{j}\not\subseteq\s$. Clearly, $\{1,2,\ldots,j\}\subseteq\di{\s}{i}$ and $\X_{j}\not\subseteq\di{\s}{i}$ for each $1\le j<i$. 
			Since $\di{\s}{i}\in\C_1$, applying the induction hypothesis iteratively to $\di{\s}{i}$, we have $\di{\s}{i}\in\C_{j+1}$ for all $1\le j<i$. Taking $j=i-1$, we get $\di{\s}{i}\in\C_i$.
			Since $\un{\s}{i}=\s\in\C_i$, Remark \ref{remark:C_i} \ref{in C_i to C_i+1} implies $\s\notin\C_{i+1}$, a contradiction. Hence $\X_i\not\subseteq\s$.
			
			\item Conversely, assume $\{1,2,\ldots,i\}\subseteq\s$ and $\X_i\not\subseteq\s$. Since $1<i\le p$, we have $i\nsim u$ if and only if $u\in\X_i$, and $i\sim0$. 
			Hence there exists $v\in\di{\s^c}{0}$ such that $i\nsim v$, which implies that $\{i,v\}$ is a disconnected $2$-set in $((\di{\s}{i})\cup\{0\})^c$. So, $\un{(\di{\s}{i})}{0}\in\Dt{2}{\c}$. 
			Therefore, $\di{\s}{i}\notin\C_1$ by Proposition \ref{proposition:first matching} \ref{sigma (not) in C iff}. Since $\C_i\subseteq\C_1$ by Remark \ref{remark:C_i} \ref{C_i subseteq}, $\di{\s}{i}\notin\C_i$. By Remark \ref{remark:C_i} \ref{in C_i to C_i+1}, $\s\in\C_i$ implies $\s\in\C_{i+1}$.  
		\end{itemize}    
		This completes the proof.
	\end{proof} 
	
	Although Proposition \ref{proposition:C_i and C_i+1 for 1<=i<=p} holds for all $n\ge 2p+2$, we apply it here only to the case $n=2p+2$. For $n\ge 3p+1$, we employ a different approach.
	
	\begin{lemma}\label{lemma:n=2p+2} 
		Let $\t=\{1,2,\ldots,p\}$. For $n=2p+2$, $\C_{p+1}=\{\t\}$.
	\end{lemma}
	
	\begin{proof} 
		By the definition of $C_{2p+2}^p$, $\{0,p+1\}$ is the only disconnected $2$-set containing $0$ in $V(C_{2p+2}^p)$.
		
		First, assume that $\s\in\C_{p+1}$. We prove $\s=\t$. By Remark \ref{remark:C_i} \ref{C_i subseteq}, $\C_{p+1} \subseteq \C_p\subseteq\ldots\subseteq \C_1\subseteq\C_0=\Dt{2}{\c}$. Hence $\s\in\C_i$ and $\s\in\C_{i+1}$ for all $1\le i \le p$. 
		By Proposition \ref{proposition:C_i and C_i+1 for 1<=i<=p}, for each $1\leq i \leq p$, we obtain $\{1,2,\ldots,i\}\subseteq\s$ and $i+p+1\notin\s$. It follows that $\{1,2,\ldots,p\}\subseteq\s$ and $\{p+2,p+3,\ldots,2p+1\}\subseteq\s^c$. 
		Since $\s\in\C_1$, Proposition \ref{proposition:first matching} \ref{sigma (not) in C iff} implies that every disconnected $2$-set in $\s^c$ contains $0$. Hence $\{0,p+1\}\subseteq\s^c$. This means that $\s=\{1,2,\ldots,p\}$, and thus $\s=\t$.
		
		We now prove that $\t\in\C_{p+1}$. Observe that $\{0,p+1\}$ is the only disconnected $2$-set in $\t^c$. Hence $\t\in\Dt{2}{C_{2p+2}^p}$ and $\un{\t}{0}\notin\Dt{2}{C_{2p+2}^p}$. By Proposition \ref{proposition:first matching} \ref{sigma (not) in C iff}, $\t\in\C_1$. 
		Since $\{1,2,\ldots,p\}\subseteq\t$ and $\{p+2,p+3,\ldots,2p+1\}\subseteq\t^c$, it follows that $\{1,2,\ldots,i\}\subseteq\t$ and $i+p+1\notin\t$ for each $1\leq i\leq p$. Applying Proposition \ref{proposition:C_i and C_i+1 for 1<=i<=p} iteratively, we obtain $\t\in\C_{p+1}$. 
		
		Hence $\C_{p+1}=\{\t\}$.
	\end{proof}
	
	We have set up the necessary results for $n=2p+2$. Now, 
	we address the case $n\geq 3p+1$. 
	For $\s\in\Dt{2}{\c}$, define $$\m{\s}:=\max(\s^c).$$ 
	By Remark \ref{remark:face of 2-cut complex}, if $\s\in\Dt{2}{\c}$, then $\c[\s^c]$ is not a complete graph. From the definition of $\c$, we then obtain the following.
	
	\begin{remark}\label{remark:range of m_sigma}
		Let $\s\in\Dt{2}{\c}$. Then $\m{\s}\in\{p+1,p+2,\ldots,n-1\}$.
	\end{remark}
	
	For $\s \in\Dt{2}{\c}$, we have $\m{\s}-p\ge1$ by Remark \ref{remark:range of m_sigma}. From Remark \ref{remark:C_i} \ref{C_i subseteq}, $\C_{\m{\s}-p}\subseteq \C_1$. The proposition below establishes the reverse inclusion $\C_1\subseteq\C_{\m{\s}-p}$.
	
	\begin{proposition}\label{proposition:n>=3p+1 and sigma in C_m-p}
		Let $n\ge 3p+1$. If $\s\in\C_1$, then $\s^c\subseteq\{0,\m{\s}-p, \m{\s}-p+1,\ldots,\m{\s}\}$ and $\s\in\C_{\m{\s}-p}$. 
	\end{proposition} 
	
	\begin{proof}
		Suppose $\s\in\C_1$. We first show that $\s^c\subseteq\{0,\m{\s}-p, \m{\s}-p+1,\ldots,\m{\s}\}$. Since $\C_1\subseteq\Dt{2}{\c}$, we have $\s\in\Dt{2}{\c}$. Hence $\m{\s}\in\{p+1,p+2,\ldots,n-1\}$ by Remark \ref{remark:range of m_sigma}. Clearly, $\m{\s}\ne0$.
		By Proposition \ref{proposition:first matching} \ref{sigma (not) in C iff}, every disconnected $2$-set in $\s^c$ contains $0$ (as $\s\in\C_1$). Therefore, since $\m{\s}\in\s^c$, it follows that if there exists $v\in V(\c)$ such that $v\ne0$ and $v\nsim\m{\s}$, then $v\in\s$. We consider the following cases depending on the value of $\m{\s}$:
		\begin{itemize}
			\item Suppose $\m{\s}\in\{p+1,p+2,\ldots,n-p-1\}$. Then $1\le\m{\s}-p<\m{\s}<\m{\s}+p\le n-1$. Since $\max(\s^c)=\m{\s}$, it follows that $\{\m{\s}+1,\m{\s}+2,\ldots,\m{\s}+p\}\subseteq\s$.
			Let $u\in\{1,2,\ldots,\m{\s}-p-1\}\cup\{\m{\s}+p+1,\m{\s}+p+2,\ldots,n-1\}$. Then $u\ne0$ and $u\nsim\m{\s}$, which implies $u\in\s$. Hence $\s^c\subseteq\{0,\m{\s}-p,\m{\s}-p+1,\ldots,\m{\s}\}$.
			
			\item Suppose $\m{\s}=n-p$. Then $1<n-2p=\m{\s}-p<\m{\s}<\m{\s}+p-1=n-1$ and $\m{\s}+p\ (\text{mod $n$})=0$. It follows that $\{\m{\s}+1,\m{\s}+2,\ldots,n-1\}\subseteq\s$.
			Now, let $u\in\{1,2,\ldots,\m{\s}-p-1\}$. Observe that $u\ne0$ and $u\nsim\m{\s}$. Hence $u\in\s$, which implies $\{1,2,\ldots,\m{\s}-p-1\}\subseteq\s$. Therefore $\s^c\subseteq\{0,\m{\s}-p,\m{\s}-p+1,\ldots,\m{\s}\}$.
			
			\item Suppose $\m{\s}\in\{n-p+1,n-p+2,\ldots,n-1\}$. Since $\s\in\Dt{2}{\c}$ and every disconnected $2$-set in $\s^c$ contains $0$, there exists $w\in\s^c$ such that $w\nsim0$. By the definition of $\c$, $w\in\{p+1,p+2,\ldots,n-p-1\}$. 
			
			Observe that $1\le\m{\s}+p\ (\text{mod $n$})<p+1\le w\le n-p-1<\m{\s}\le n-1$. Since $n\ge3p+1$, we have $\m{\s}+p\ (\text{mod $n$})<\m{\s}-2p$. Moreover,  $p\ge1$ implies $\m{\s}-2p<\m{\s}-p<\m{\s}$.
			Therefore $\m{\s}+p\ (\text{mod $n$})<w,\m{\s}-p<\m{\s}$. If $w<\m{\s}-p$, then $w\nsim\m{\s}$, which implies $w\in\s$ (as $w\ne0$), a contradiction. So $w\ge\m{\s}-p$. 
			
			It follows that $1\le\m{\s}+p\ (\text{mod $n$})<\m{\s}-2p\le w-p<w<w+p\le n-1$. 
			Let $u\in\{1,2,\ldots,\m{\s}+p\ (\text{mod $n$})\}$. Then $u\ne0$ and $w\nsim u$, which implies $\{u,w\}$ is a disconnected $2$-set not containing $0$. Since $w\in\s^c$, we get $u\in\s$, and thus $\{1,2,\ldots,\m{\s}+p\ (\text{mod $n$})\}\subseteq\s$. 
			Now, let $u'\in\{\m{\s}+p\ (\text{mod $n$})+1,\m{\s}+p\ (\text{mod $n$})+2,\ldots,\m{\s}-p-1\}$. Then $u'\ne0$ and $u'\nsim\m{\s}$, which implies $u'\in\s$. Hence $\{\m{\s}+p\ (\text{mod $n$})+1,\m{\s}+p\ (\text{mod $n$})+2,\ldots,\m{\s}-p-1\}\subseteq\s$. Moreover, $\max(\s^c)=\m{\s}$ implies that $\{\m{\s}+1,\m{\s}+2,\ldots,n-1\}\subseteq\s$. This means that $\s^c\subseteq\{0,\m{\s}-p,\m{\s}-p+1,\ldots,\m{\s}\}$.
		\end{itemize}
		Therefore, $\s^c\subseteq\{0,\m{\s}-p, \m{\s}-p+1,\ldots,\m{\s}\}$.
		
		We now prove that $\s\in\C_{\m{\s}-p}$. If $\m{\s}=p+1$, then $\m{\s}-p=1$. Since $\s\in\C_1$, we are done.
		
		Now, suppose $\m{\s}\in\{p+2,p+3,\ldots,n-1\}$. Let $1\le i\le \m{\s}-p-1$. Since $\s^c\subseteq\{0,\m{\s}-p, \m{\s}-p+1,\ldots,\m{\s}\}$, $i\in\s$. We first show that $(\di{\s}{i})\cup\{0\}\in\Dt{2}{\c}$. 
		Note that $((\di{\s}{i})\cup\{0\})^c=\un{(\un{\s}{0})^c}{i}$. By Remark \ref{remark:face of 2-cut complex}, it suffices to find $v\in(\un{\s}{0})^c$ such that $i\nsim v$.
		\begin{itemize}
			\item Suppose $\m{\s}\in\{p+2,p+3,\ldots,n-p\}$. Then $2\leq \m{\s}-p<\m{\s}<\m{\s}+p\leq n$. Since $1\le i\le \m{\s}-p-1$, by the definition of $\c$, $i\nsim \m{\s}$. Therefore, $v=\m{\s}$ such that $i\nsim v$. 
			
			\item Suppose $\m{\s}\in\{n-p+1,n-p+2,\ldots,n-1\}$. Since $n\ge3p+1$ and $p\ge 1$, we have $\m{\s}+p\ (\text{mod $n$})<\m{\s}-2p<\m{\s}-p$. 
			So $1\le \m{\s}+p\ (\text{mod $n$})\le\m{\s}-p-1$. Since $1\le i\le \m{\s}-p-1$, we have $1\le i\leq \m{\s}+p\ (\text{mod $n$})$ or $\m{\s}+p\ (\text{mod $n$})<i\le \m{\s}-p-1$. 
			
			If $\m{\s}+p\ (\text{mod $n$})<i<\m{\s}-p$, then $i\nsim \m{\s}$. Therefore, we have $v=\m{\s}$ with $i\nsim v$. 
			
			Now, assume that $1\le i\leq \m{\s}+p\ (\text{mod $n$})$. 
			Since $\s\in\C_1$, it follows by Proposition \ref{proposition:first matching} \ref{sigma (not) in C iff} that every disconnected $2$-set in $\s^{c}$ contains $0$. Further, $\s \in\C_1\subseteq\Dt{2}{\c}$ implies that there exists $u\in\di{\s^c}{0}=(\un{\s}{0})^c$ such that $0\nsim u$.
			By definition of $\c$, we have $p+1\le u\le n-p-1$. Moreover, $\s^c\subseteq\{0,\m{\s}-p,\m{\s}-p+1,\ldots,\m{\s}\}$ implies that $\m{\s}-p\le u\le\m{\s}$. This means that $\m{\s}-p\leq u\leq n-p-1$, and thus $\m{\s}-2p\leq u-p<u<u+p\leq n-1$. 
			Since $1\leq \m{\s}+p\ (\text{mod $n$})<\m{\s}-2p$ and $1\leq i\leq \m{\s}+p\ (\text{mod $n$})$, we get $i\notin\{u-p,u-p+1,\ldots,u,\ldots,u+p\}$, and hence $i\nsim u$. Thus, we have $v=u$ with $i\nsim v$.    
		\end{itemize} 
		In all cases, $(\di{\s}{i})\cup\{0\}\in\Dt{2}{\c}$. Hence $\di{\s}{i}\notin\C_1$ for all $1\le i\le\m{\s}-p-1$ by Proposition \ref{proposition:first matching} \ref{sigma (not) in C iff}, and therefore $\s\in\C_{\m{\s}-p}$ by Remark \ref{remark:C_i} \ref{in C_j+1} (as $\s\in\C_1$).
		
		This completes the proof.
	\end{proof}
	
	Proposition \ref{proposition:n>=3p+1 and sigma in C_m-p} states that if $\s\in\C_1$, then $\s\in\C_{\m{\s}-p}$. The following result provides a sufficient condition for when $\s\notin\C_{\m{\s}-p+1}$.
	
	\begin{proposition}\label{proposition:matched in m-p}
		Let $n\ge 3p+1$ and let $\s\in\C_1$. Suppose that there exists a disconnected $2$-set $\{0,v\}$ in $\s^c$ such that $v\neq \m{\s}-p$. Then $\s\notin\C_{\m{\s}-p+1}$.  
	\end{proposition}
	
	\begin{proof}
		Since $\s\in\C_1$, it follows by Proposition \ref{proposition:n>=3p+1 and sigma in C_m-p} that $\s\in\C_{\m{\s}-p}$ and $\s^c\subseteq\{0,\m{\s}-p,\m{\s}-p+1,\ldots,\m{\s}\}$. We have either $\m{\s}-p\in\s$ or $\m{\s}-p\notin\s$.
		
		First, suppose that $\m{\s}-p\in\s$. Then $\m{\s}-p\notin(\un{\s}{0})^c$. Hence $(\un{\s}{0})^c\subseteq\{\m{\s}-p+1,\m{\s}-p+2,\ldots,\m{\s}\}$, and thus $\m{\s}-p\sim u$ for all $u\in(\un{\s}{0})^c$. 
		Since $\s\in\C_1$, we get $\di{\s}{\m{\s}-p}\in\C_1$ by Proposition \ref{proposition:first matching} \ref{sigma minus u in C}. Observe that $\m{(\di{\s}{\m{\s}-p})}=\m{\s}$. 
		By Proposition \ref{proposition:n>=3p+1 and sigma in C_m-p}, $\di{\s}{\m{\s}-p}\in\C_1$ implies $\di{\s}{\m{\s}-p}\in\C_{\m{\s}-p}$. Since $\un{\s}{\m{\s}-p}=\s\in\C_{\m{\s}-p}$, we have $\s\notin\C_{\m{\s}-p+1}$ by Remark \ref{remark:C_i} \ref{in C_i to C_i+1}.
		
		Now, let $\m{\s}-p\notin\s$. Since there exists a disconnected $2$-set $\{0,v\}$ in $\s^c$ such that $v\neq \m{\s}-p$, we get $\s\cup\{\m{\s}-p\}\in\Dt{2}{\c}$.  
		By Proposition \ref{proposition:first matching} \ref{sigma union u in C}, $\s\cup\{\m{\s}-p\}\in\C_1$ (as $\s\in\C_1$). Since $\m{\s}-p\neq\m{\s}$, we get $\m{(\un{\s}{\m{\s}-p})}=\m{\s}$. 
		It follows that $\s\cup\{\m{\s}-p\}\in\C_{\m{\s}-p}$ by Proposition \ref{proposition:n>=3p+1 and sigma in C_m-p}. We have $\di{\s}{\m{\s}-p}=\s\in\C_{\m{\s}-p}$. By Remark \ref{remark:C_i} \ref{in C_i to C_i+1}, $\s\notin\C_{\m{\s}-p+1}$.
	\end{proof}
	
	\begin{proposition}\label{proposition:in C_n-p}
		Let $n\ge 3p+1$. For $\s\in\Dt{2}{\c}$, we have the following. 
		\begin{enumerate}[label=(\roman*)]
			\item \label{short range of m_sigma} If $\{0,\m{\s}-p\}$ is the only disconnected $2$-set in $\s^c$, then $\m{\s}\in\{n-p,n-p+1,\ldots,n-1\}$. 
			\item \label{C_n-p} $\s\in\C_{n-p}$ \iff $\{0,\m{\s}-p\}$ is the only disconnected $2$-set in $\s^c$.
		\end{enumerate}
	\end{proposition}
	
	\begin{proof}
		\begin{enumerate}[label=(\roman*)]        
			\item By Remark \ref{remark:range of m_sigma}, $\m{\s}\in\{p+1,p+2,\ldots,n-1\}$. If $\m{\s}\in\{p+1,p+2,\ldots,n-p-1\}$, then $0\nsim \m{\s}$. Since $\{0,\m{\s}\}\subseteq\s^c$, we get a contradiction to the fact that $\{0,\m{\s}-p\}$ is the only disconnected $2$-set in $\s^c$. Hence $\m{\s}\in\{n-p,n-p+1,\ldots,n-1\}$.
			
			\item Suppose $\s\in\C_{n-p}$. Since $p+1\le\m{\s}\leq n-1$, we have $1<\m{\s}-p+1\le n-p$. Hence $\C_{n-p}\subseteq\C_{\m{\s}-p+1}\subseteq\C_1$ implies that $\s\in\C_{\m{\s}-p+1}$ and $\s\in\C_1$. Now, $\s\in\C_1$ implies that every disconnected $2$-set in $\s^{c}$ contains $0$. 
			If $\{0,v\}$ is a disconnected $2$-set in $\s^c$ with $v\neq \m{\s}-p$, then $\s\notin\C_{\m{\s}-p+1}$ by Proposition \ref{proposition:matched in m-p}, a contradiction. Therefore, $\{0,\m{\s}-p\}$ is the only disconnected $2$-set in $\s^c$. 
			
			Conversely, assume that $\{0,\m{\s}-p\}$ is the only disconnected $2$-set in $\s^c$. By Remark \ref{remark:face of 2-cut complex}, $\un{\s}{0}$ and $\un{\s}{\m{\s}-p}\notin\Dt{2}{\c}$. Hence $\s\in\C_1$, and thus $\s\in\C_{\m{\s}-p}$ by Proposition \ref{proposition:n>=3p+1 and sigma in C_m-p}. 
			Moreover, since $\C_{\m{\s}-p}\subseteq\Dt{2}{\c}$, $\un{\s}{\m{\s}-p}\notin\C_{\m{\s}-p}$. Therefore, $\s\in\C_{\m{\s}-p+1}$ by Remark \ref{remark:C_i} \ref{in C_i to C_i+1}.
			
			Since $\{0,\m{\s}-p\}$ is the only disconnected $2$-set in $\s^c$, $\m{\s}\in\{n-p,n-p+1,\ldots,n-1\}$ by \ref{short range of m_sigma}. If $\m{\s}=n-1$, then $\m{\s}-p+1=n-p$, and thus $\s\in\C_{\m{\s}-p+1}$ implies that $\s\in\C_{n-p}$, as required.  
			
			Now, assume that $\m{\s}\in\{n-p,n-p+1,\ldots,n-2\}$. Let $\m{\s}-p+1\le j\le n-p-1$. Then $\m{\s}\ge n-p$ and $n\ge 3p+1$ implies that $p+1<j<n-p$, and hence $j\nsim0$. Since $\{0,\m{\s}-p\}$ is the only disconnected $2$-set in $\s^c$, we have $j\in\s$.    
			By Proposition \ref{proposition:n>=3p+1 and sigma in C_m-p}, $\s\in\C_1$ implies that $\s^c\subseteq\{0,\m{\s}-p,\m{\s}-p+1,\ldots,\m{\s}\}$. Since $\m{\s}-p<\m{\s}-p+1\le j\le n-p-1<\m{\s}$, it follows that $j\sim u$ for all $u\in(\un{\s}{0})^c$. 
			Therefore, $\s\in\C_1$ implies that $\di{\s}{j}\in\C_1$ by Proposition \ref{proposition:first matching} \ref{sigma minus u in C}. Note that $\m{(\di{\s}{j})}=\m{\s}$ (as $j<\m{\s}$). 
			Moreover, $\{0,j\}\subseteq(\di{\s}{j})^c$ such that $0\nsim j$ and $j\neq \m{\s}-p$. Hence $\di{\s}{j}\notin\C_{\m{\s}-p+1}$ by Proposition \ref{proposition:matched in m-p}. Since $\s\in\C_{\m{\s}-p+1}$ and $\m{\s}-p+1\le j\le n-p-1$, we get $\s\in\C_{n-p}$ by Remark \ref{remark:C_i} \ref{in C_j+1}. 
		\end{enumerate}
	\end{proof}
	
	\begin{lemma}\label{lemma:n>=3p+1}
		Let $n\geq 3p+1$. Then $\C_{n-p+1}=\{\t\}$, where $\t^c=\{0,n-2p,n-p\}$.
	\end{lemma}
	
	\begin{proof}
		Since $n\geq3p+1$, $0\nsim n-2p$. Hence $\t \in\Dt{2}{\c}$. We first show that $\t\in\C_{n-p+1}$. Observe that $\{0,\m{\t}-p\}=\{0,n-2p\}$ is the only disconnected $2$-set in both $\t^c$ and $(\un{\t}{n-p})^c$. 
		Therefore, $\t\in\C_{n-p}$ by Proposition \ref{proposition:in C_n-p} \ref{C_n-p}. Clearly, $\un{\t}{n-p}\in\Dt{2}{\c}$ and $\un{\t}{0,n-p}\notin\Dt{2}{\c}$. Hence $\un{\t}{n-p}\in\C_1$.  
		Since $\{0,n-2p\}$ is a disconnected $2$-set in $(\un{\t}{n-p})^c$ with $n-2p=\m{(\un{\t}{n-p})}\neq\m{(\un{\t}{n-p})}-p$, it follows from Proposition \ref{proposition:matched in m-p} that $\un{\t}{n-p}\notin\C_{n-3p+1}$.
		Further, $\C_{n-p}\subseteq\C_{n-3p+1}$ (as $n-3p+1<n-p$) implies that $\un{\t}{n-p}\notin\C_{n-p}$. Therefore, $\t\in\C_{n-p}$ implies $\t\in\C_{n-p+1}$ by Remark \ref{remark:C_i} \ref{in C_i to C_i+1}.
		
		Let $\s\in\C_{n-p+1}$. We prove that $\s=\t$. Since $\C_{n-p+1}\subseteq\C_{n-p}\subseteq\ldots\subseteq\C_1$, $\s\in\C_1\cap\C_{n-p}$. 
		From Proposition \ref{proposition:n>=3p+1 and sigma in C_m-p}, $\s^c\subseteq\{0,\m{\s}-p,\m{\s}-p+1,\ldots,\m{\s}\}$, and by Proposition \ref{proposition:in C_n-p} \ref{C_n-p}, $\{0,\m{\s}-p\}$ is the only disconnected $2$-set in $\s^c$. Hence $\m{\s}\in\{n-p,n-p+1,\ldots,n-1\}$ by Proposition \ref{proposition:in C_n-p} \ref{short range of m_sigma}.  
		
		Suppose $\m{\s}\in\{n-p+1,n-p+2,\ldots,n-1\}$. Then $\m{\s}-p<n-p<\m{\s}$. Since $n-p\sim0$ and $\s^c\subseteq\{0,\m{\s}-p,\m{\s}-p+1,\ldots,\m{\s}\}$, it follows that $n-p\sim u$ for all $u\in\s^c$. 
		Therefore, $\{0,\m{\s}-p\}$ is the only disconnected $2$-set in $\s^c$ implies that $\{0,\m{\s}-p\}$ is the only disconnected $2$-set in both $(\di{\s}{n-p})^c$ and $(\un{\s}{n-p})^c$. 
		Moreover, $\m{(\di{\s}{n-p})}=\m{(\un{\s}{n-p})}=\m{\s}$. By Proposition \ref{proposition:in C_n-p} \ref{C_n-p}, $\di{\s}{n-p}$, $\un{\s}{n-p}\in\C_{n-p}$. Hence $\s\notin\C_{n-p+1}$ by Remark \ref{remark:C_i} \ref{in C_i to C_i+1}, a contradiction. So $\m{\s}=n-p$.   
		
		Since $\s^c\subseteq \{0,\m{\s}-p,\m{\s}-p+1,\ldots,\m{\s}\}$, we get $\s^c\subseteq\{0,n-2p,n-2p+1,\ldots,n-p\}$. Moreover, $p\geq 1$ and $n\geq 3p+1$ imply $p+1\le n-2p\leq n-p-1$. 
		This means that $0\nsim u$ for all $u\in\s^c\setminus\{0,n-p\}$. We have $\{0,\m{\s}-p\}=\{0,n-2p\}$ is the only disconnected $2$-set in $\s^c$ and $\m{\s}=n-p\in\s^c$. It follows that $\s^c=\{0,n-2p,n-p\}$. Therefore $\s=\t$.
	\end{proof}
	
	\begin{theorem}\label{theorem:powers_of_cycle} 
		Let $n\geq 2p+2$. Then
		$\Dt{2}{\c}\simeq 
		\begin{cases}
			\mathbb{S}^{\frac{n-4}{2}} & \text{ if }n=2p+2,\\
			\mathbb{S}^{n-4} & \text{ if }n\ge 3p+1.
		\end{cases}$
	\end{theorem}
	
	\begin{proof}
		By our construction, $\M=\bigsqcup_{i=0}^{n-1} \M_i$ is an acyclic matching on $\Dt{2}{\c}$ with $\C_n$ as the set of all critical faces. Observe that for any $0\le i\le n-1$, if $\C_i$ is a singleton set, then $\C_i=\C_n$.
		\begin{itemize}
			\item Let $n=2p+2$. Then Lemma \ref{lemma:n=2p+2} implies $\C_{p+1}=\{\s\}$, where $\s=\{1,2,\ldots,p\}$. Therefore $\C_n=\{\s\}$. Note that the dimension of $\s$ is $p-1=\dfrac{n-4}{2}$.
			
			\item Let $n\ge 3p+1$. By Lemma \ref{lemma:n>=3p+1}, $\C_{n-p+1}=\{\t\}$, where $\t^c=\{0,n-2p,n-p\}$. Hence $\C_n=\{\t\}$. Moreover, the dimension of $\t$ is $n-4$.
		\end{itemize}
		
		Result follows from Theorem \ref{theorem:acyclic} and Corollary \ref{corollary:critical}.
	\end{proof}
	
	\subsection{Cartesian product of complete graphs and path graphs} \label{subsection:Km_Pn}
	In this section, we prove Theorem \ref{theorem:Intro_Cartesian_path}. For a positive integer $q$, let $[q]=\{1,2,\ldots,q\}$ and $[0,q]=[q]\cup\{0\}$. The Cartesian product $K_m \square P_n$ has vertex set
	$$V(K_m \square P_n)=V(K_m)\times V(P_n)=\{(i,j)\, |\, i\in[0,m-1]\text{ and } j\in[0,n-1]\}.$$
	For any distinct $(i_1,j_1),\,(i_2,j_2)\in V(K_m \square P_n)$, we have $(i_1,j_1)\sim(i_2,j_2)$ \iff $i_1=i_2$ and $|j_1-j_2|=1$, or $j_1=j_2$. Equivalently, $(i_1,j_1)\nsim(i_2,j_2)$ \iff $i_1= i_2$ and $|j_1-j_2|>1$, or $i_1\neq i_2$ and $j_1\neq j_2$.
	
	For notational convenience, we denote the graph $K_m \square P_n$ by $\g$, and write any vertex of the form $(0,j) \in V(\g)$ simply as $(j)_0$.  
	We now define a sequence of element matchings $\M_{(0)_0},\M_{(1)_0},\ldots,$ $\M_{(n-1)_0}$ on $\Dt{2}{\g}$ and the sets $\C_0,\C_1,\ldots,\C_{n}$, using vertices $(0)_0,(1)_0,\ldots,(n-1)_0$, as in Equation \eqref{equation:sequence_of_element_matchings}.
	By Lemma \ref{lemma:Jonsson acyclic}, $\M:=\bigsqcup_{i=0}^{n-1}\M_{(i)_0}$ is an acyclic matching on $\Dt{2}{\g}$. Clearly, $\C_0=\Dt{2}{\g}$ and $\C_1=\{\s\in\Dt{2}{\g}\, |\, \s\notin\M_{(0)_0}\}$.
	
	\begin{proposition}\label{proposition:Km_Pn not in C_n}
		Let $\s\in\C_1$.
		\begin{enumerate}[label=(\roman*)]
			\item \label{Km_Pn:0 nsim 0} If $\s^c=\{(0)_0,(j)_0\}$ for some $j\in[n-1]$, then $\s,\,\di{\s}{(j-1)_0}\notin\C_{j}$, and hence $\s,\,\di{\s}{(j-1)_0}\notin\C_{n}$.
			
			\item \label{Km_Pn:0 nsim ij} If $\{(1)_0,(2)_0,\ldots,(j)_0\}\subseteq\s,\,\{(0)_0,(i,j)\}\subseteq\s^c$ and $\di{\s}{(j)_0}\in\C_1$ for some $i\in[m-1],\,j\in[n-1]$, then $\s,\,\di{\s}{(j)_0}\notin\C_{n}$.
		\end{enumerate}
	\end{proposition}
	
	\begin{proof}
		\begin{enumerate}[label=(\roman*)]
			\item Suppose $\s^c=\{(0)_0,(j)_0\}$. Since $\s\in\C_1\subseteq\Dt{2}{\g}$, $(0)_0\nsim(j)_0$ by Remark \ref{remark:face of 2-cut complex}. Hence $j \in [n-1] \setminus \{1\}$. Since $\C_n\subseteq\C_{j}$, it suffices to prove $\s,\,\di{\s}{(j-1)_0}\notin\C_{j}$. 
			
			We have $(\un{\s}{(0)_0})^c=\{(j)_0\}$ and $(j-1)_0\sim(j)_0$. By Proposition \ref{proposition:first matching} \ref{sigma minus u in C}, $\di{\s}{(j-1)_0}\in\C_1$.
			
			If $j=2$, then $\un{\s}{(1)_0}=\s\in\C_1$ and $\di{\s}{(1)_0}\in\C_1$ imply $\s,\,\di{\s}{(1)_0}\notin\C_2=\C_j$ by Remark \ref{remark:C_i} \ref{in C_i to C_i+1}.  
			
			Now, assume $3\le j\le n-1$. Let $1\le r\le j-2$. Then $(r)_0\in\s,\,\di{\s}{(j-1)_0}$. Since $j-r\ge2$, $(j)_0\nsim(r)_0$. This implies $\un{(\di{\s}{(r)_0})}{(0)_0},\,\un{(\di{\s}{(r)_0,(j-1)_0})}{(0)_0}\in\Dt{2}{\g}$. 
			Hence $\di{\s}{(r)_0},\,\di{\s}{(r)_0,(j-1)_0}\notin\C_1$ by Proposition \ref{proposition:first matching} \ref{sigma (not) in C iff}. From Remark \ref{remark:C_i} \ref{in C_j+1}, we obtain $\s,\,\di{\s}{(j-1)_0}\in\C_{j-1}$. Since $\un{\s}{(j-1)_0}=\s$, Remark \ref{remark:C_i} \ref{in C_i to C_i+1} implies $\s,\,\di{\s}{(j-1)_0}\notin\C_j$.
			
			\item Let $i\in[m-1],\,j\in[n-1]$ and suppose that $\{(1)_0,(2)_0,\ldots,(j)_0\}\subseteq\s,\,\{(0)_0,(i,j)\}\subseteq\s^c$ and $\di{\s}{(j)_0}\in\C_1$. Since $\C_n\subseteq\C_{j+1}$ by Remark \ref{remark:C_i} \ref{C_i subseteq}, we prove $\s,\,\di{\s}{(j)_0}\notin\C_{j+1}$.
			
			For $j=1$, $\un{\s}{(1)_0}=\s\in\C_1$ and $\di{\s}{(1)_0}\in\C_1$ imply $\s,\,\di{\s}{(1)_0}\notin\C_2=\C_{j+1}$. 
			
			Now, suppose $2\le j\le n-1$. Let $1\le r\le j-1$. Then $(r)_0\in\s,\,\di{\s}{(j)_0}$ and $(r)_0\nsim (i, j)$ (as $i\neq 0$ and $j\neq r$). Hence $\un{(\di{\s}{(r)_0})}{(0)_0}$, $\un{(\di{\s}{(r)_0,(j)_0})}{(0)_0}\in\Dt{2}{\g}$. 
			From Proposition \ref{proposition:first matching} \ref{sigma (not) in C iff}, $\di{\s}{(r)_0}$, $\di{\s}{(r)_0,(j)_0}\notin\C_1$. Therefore $\s,\,\di{\s}{(j)_0}\in\C_{j}$ by Remark \ref{remark:C_i} \ref{in C_j+1}, and thus $\s,\,\di{\s}{(j)_0}\notin\C_{j+1}$ by Remark \ref{remark:C_i} \ref{in C_i to C_i+1}. 
		\end{enumerate}
	\end{proof}
	
	\begin{theorem}\label{theorem:K_m[]P_n}
		For $m,n\geq2$, $\Dt{2}{K_m \square P_n}\simeq\bigvee_{(m-1)(n-1)}\bb{S}^{mn-4}.$
	\end{theorem}
	
	\begin{proof}
		We first prove that $\s\in\C_n$ \iff $\s^c=\{(0)_0,(i,j-1),(i,j)\}$ for some $i\in[m-1]$ and $j\in[n-1].$ 
		
		First, assume $\s\in\C_n$. Since $\C_{n}\subseteq \C_1\subseteq\Dt{2}{\g}$, $\s\in\C_1\cap\Dt{2}{\g}$. 
		By Proposition \ref{proposition:first matching} \ref{sigma (not) in C iff}, every disconnected $2$-set in $\s^c$ contains $(0)_0$. It follows that there $\s^c=\{(0)_0,(i,j)\}\cup X$ for some $(i,j)\in V(\g)$ and $X\subseteq V(\g)$ such that $(i,j)\nsim(0)_0$ and $\g[\un{X}{(i,j)}]$ is a complete graph. Since $(i,j)\in V(\g)$, we have $i\in[0,m-1]$.
		\begin{itemize}
			\item Suppose $i=0$. Then $(i,j)=(j)_0\nsim(0)_0$ implies $j\in\di{[n-1]}{1}$. Since $\g[\un{X}{(j)_0}]$ is a complete graph, either (i) $X=\emptyset$, or (ii) $X=\{(j-1)_0\}$, or (iii) $X\subseteq\{(i',j)\, |\, i'\in[m-1]\}$ with $X\ne\emptyset$. 
			
			We have $\s^c=\{(0)_0,(j)_0\}\cup X$. If $X=\emptyset$, then $\s^c=\{(0)_0,(j)_0\}$, and if $X=\{(j-1)_0\}$, then $(\un{\s}{(j-1)_0})^c=\{(0)_0,(j)_0\}$. In either case, Proposition \ref{proposition:Km_Pn not in C_n} \ref{Km_Pn:0 nsim 0} contradicts $\s\in\C_n$. 
			
			Now, assume $X\subseteq\{(i',j)\, |\, i'\in[m-1]\}$ with $X\ne\emptyset$. Then there exists $t\in[m-1]$ such that $(t,j)\in X$. Observe that $\{(1)_0,(2)_0,\ldots,(j)_0\}\subseteq\un{\s}{(j)_0}$ and $\{(0)_0,(t,j)\}\subseteq(\un{\s}{(j)_0})^c$. 
			Since $\s\in\C_1$, applying Proposition \ref{proposition:Km_Pn not in C_n} \ref{Km_Pn:0 nsim ij} to $\un{\s}{(j)_0}$ gives $\s\notin\C_n$, again a contradiction. Hence $i\ne0$.
			
			\item Suppose $i\in[m-1]$. Then $(i,j)\nsim(0)_0$ implies $j\in[n-1]$. Since $\g[\un{X}{(i,j)}]$ is a complete graph, either (i) $X=\emptyset$, or (ii) $X=\{(i,j-1)\}$, or (iii) $X\subseteq\{(i',j)\, |\, i'\in\di{[0,m-1]}{i}\}$ with $X\ne\emptyset$.
			
			If $X=\emptyset$, then $\s^c=\{(0)_0,(i,j)\}$ and $\{(1)_0,(2)_0,\ldots,(j)_0\}\subseteq\s$. Since $\s\in\C_1$, $(\un{\s}{(0)_0})^c=\{(i,j)\}$ and $(j)_0\sim(i,j)$, we have $\di{\s}{(j)_0}\in\C_1$ by Proposition \ref{proposition:first matching} \ref{sigma minus u in C}. Hence $\s\notin\C_n$ by Proposition \ref{proposition:Km_Pn not in C_n} \ref{Km_Pn:0 nsim ij}, a contradiction.
			
			Now, suppose $X\subseteq\{(i',j)\, |\, i'\in\di{[0,m-1]}{i}\}$ with $X\ne\emptyset$. We have $\s^c=\{(0)_0,(i,j)\}\cup X$. If $(j)_0\in X$, then $\{(1)_0,(2)_0,\ldots,(j)_0\}\subseteq\un{\s}{(j)_0}$. 
			Moreover, $\{(0)_0,(i,j)\}\subseteq(\un{\s}{(j)_0})^c$ and $\s\in\C_1$. Applying Proposition \ref{proposition:Km_Pn not in C_n} \ref{Km_Pn:0 nsim ij} to $\un{\s}{(j)_0}$, we get $\s\notin\C_n$, a contradiction. 
			If $(j)_0\notin X$, then $\{(1)_0,(2)_0,\ldots,(j)_0\}\subseteq\s$. Observe that $\g[(\un{(\di{\s}{(j)_0})}{(0)_0})^c]$ is a complete graph. 
			Hence $\di{\s}{(j)_0}\in\C_1$ by Proposition \ref{proposition:first matching} \ref{sigma minus u in C}. By Proposition \ref{proposition:Km_Pn not in C_n} \ref{Km_Pn:0 nsim ij}, $\s\notin\C_n$, a contradiction.
			
			Therefore $X=\{(i,j-1)\}$. 
		\end{itemize}
		This means that $\s^c=\{(0)_0,(i,j-1),(i,j)\}\text{ for some }i\in[m-1]\text{ and }j\in[n-1].$ 
		
		Conversely, assume that $\s^c=\{(0)_0,(i,j-1),(i,j)\}$ for some $i\in[m-1]$ and $j\in[n-1]$. Since $(0)_0\nsim(i,j)$ and $(i,j-1)\sim(i,j)$, it follows that $\s\in\Dt{2}{\g}$ and $\un{\s}{0}\notin\Dt{2}{\g}$. Therefore $\s\in\C_1$ by Proposition \ref{proposition:first matching} \ref{sigma (not) in C iff}.
		
		Let $1\le r\le n-1$. Then $(r)_0\in\s$. Observe that if $r=j$, then $(r)_0\nsim(i,j-1)$; and if $r\ne j$, then $(r)_0\nsim(i,j)$. Hence $\un{(\di{\s}{(r)_0})}{(0)_0}\in\Dt{2}{\g}$. 
		By Proposition \ref{proposition:first matching} \ref{sigma (not) in C iff}, $\di{\s}{(r)_0}\notin\C_1$. Hence $\s\in\C_{n}$ by Remark \ref{remark:C_i} \ref{in C_j+1}.
		
		We conclude that $\C_n=\{\s\in\Dt{2}{\g} \mid \s^c=\{(0)_0,(i,j-1),(i,j)\}$ for some $i\in[m-1],\,j\in[n-1]\}$.
		Note that $|\C_n|=(m-1)(n-1)$ and each $\s\in\C_n$ has dimension $mn-4$. Hence the result follows by Theorem \ref{theorem:acyclic} and Corollary \ref{corollary:critical}. 
	\end{proof}
	
	\subsection{Cartesian product of complete graphs and cycle graphs} \label{subsection:Km_Cn} In this section, we prove Theorem \ref{theorem:Intro_Cartesian_cycle}. For a positive integer $q$, we use the notation $[q]$ and $[0,q]$ from the previous section. The Cartesian product $K_m\square C_n$ has vertex set $$V(K_m\square C_n)=V(K_m) \times V(C_n)=\{(i,j) \mid i\in [0,m-1]\text{ and }j\in [0,n-1]\}.$$ 
	For any distinct $(i_1,j_1),\,(i_2,j_2)\in V(K_m\square C_n)$, $(i_1,j_1)\sim(i_2,j_2)$ \iff $i_1=i_2$ and $|j_1-j_2|\in\{1, n-1\}$, or $j_1=j_2$. Equivalently, $(i_1,j_1)\nsim(i_2,j_2)$ \iff $i_1= i_2$ and $1<|j_1-j_2|<n-1$, or $i_1\neq i_2$ and $j_1\neq j_2$. 
	
	The proof of Theorem \ref{theorem:Intro_Cartesian_cycle} follows the same approach as that of Theorem \ref{theorem:Intro_Cartesian_path}. The only difference arises from the adjacency of vertices $0$ and $n-1$ in the cycle graph $C_n$. 
	
	Let $\gp$ denote the graph $K_m \square C_n$. We write $(j)_0$ for a vertex of the form $(0, j) \in V(\gp)$. We define a sequence of element matchings $\M_{(0)_0},\M_{(1)_0},\ldots,\M_{(n-1)_0}$ on $\Dt{2}{\gp}$ and the sets $\C_0,\C_1,\ldots,\C_{n}$, using vertices $(0)_0,(1)_0,\ldots,(n-1)_0$, as in Equation \eqref{equation:sequence_of_element_matchings}.
	Then $\M:=\bigsqcup_{i=0}^{n-1}\M_{(i)_0}$ is an acyclic matching on $\Dt{2}{\gp}$. We have $\C_0=\Dt{2}{\gp}$ and $\C_1=\{\s\in\Dt{2}{\gp}\, |\, \s\notin\M_{(0)_0}\}$.
	
	The following result is analogous to Proposition \ref{proposition:Km_Pn not in C_n}.
	\begin{proposition}\label{proposition:Km_Cn not in C_n}
		Let $\s\in\C_1$.
		\begin{enumerate}[label=(\roman*)]
			\item \label{Km_Cn:0 nsim 0} If $\s^c=\{(0)_0,(j)_0\}$ for some $j\in[n-1]$, then $\s,\,\di{\s}{(j-1)_0}\notin\C_{j}$, and hence $\s,\,\di{\s}{(j-1)_0}\notin\C_{n}$.
			
			\item \label{Km_Cn:0 nsim ij} If $\{(1)_0,(2)_0,\ldots,(j)_0\}\subseteq\s,\,\{(0)_0,(i,j)\}\subseteq\s^c$ and $\di{\s}{(j)_0}\in\C_1$ for some $i\in[m-1],\,j\in[n-1]$, then $\s,\,\di{\s}{(j)_0}\notin\C_{n}$.
		\end{enumerate}
	\end{proposition}
	
	\begin{proof}
		\begin{enumerate}[label=(\roman*)]
			\item We use a similar argument as in the proof of Proposition \ref{proposition:Km_Pn not in C_n} \ref{Km_Pn:0 nsim 0}. Note that in the proof of Proposition \ref{proposition:Km_Pn not in C_n} \ref{Km_Pn:0 nsim 0}, $(0)_0\nsim(j)_0$ \iff $j\in\di{[n-1]}{1}$, and if $3\le j\le n-1$, then for any $1\le r\le j-2$, we have $j-r\ge 2$, which implies $(j)_0\nsim (r)_0$. 
			Here, $(0)_0\nsim(j)_0$ \iff $j\in\di{[n-2]}{1}$, and if $3\le j\le n-2$, then for any $1\le r\le j-2$, we have $2\le j-r\le n-3$, which implies $(j)_0\nsim (r)_0$. Using these results and following the proof of Proposition \ref{proposition:Km_Pn not in C_n} \ref{Km_Pn:0 nsim 0}, we are done.
			\item The proof is identical to the proof of Proposition \ref{proposition:Km_Pn not in C_n} \ref{Km_Pn:0 nsim ij}.
		\end{enumerate}
	\end{proof}
	
	\begin{theorem}\label{theorem:K_m[]C_n}
		For $m \geq 2$ and $n\geq 4$, $\Dt{2}{K_m\square C_n}\simeq\bigvee_{n(m-1)+1}\bb{S}^{mn-4}.$
	\end{theorem}
	
	\begin{proof}
		Let 
		\begin{align*}
			A&=\{\s\subseteq V(\gp)\,|\,\s^c=\{(0)_0,(n-2)_0,(n-1)_0\}\},\\
			B&=\{\s\subseteq V(\gp)\,|\,\s^c=\{(0)_0,(i,0),(i,n-1)\}\text{ for some }i\in[m-1]\},\\
			C&=\{\s\subseteq V(\gp)\,|\,\s^c=\{(0)_0,(i,j-1),(i,j)\}\text{ for some }i\in[m-1],\,j\in[n-1]\}.
		\end{align*}
		Note that $A,\, B$, and $C$ are pairwise disjoint. We show that $\C_n=A\sqcup B\sqcup C$. 
		
		Assume first that $\s\in\C_n$. Then $\s\in\C_1\cap\Dt{2}{\gp}$. Therefore, Proposition \ref{proposition:first matching} \ref{sigma (not) in C iff} implies that every disconnected $2$-set in $\s^c$ contains $(0)_0$.
		This means that $\s^c=\{(0)_0,(i,j)\}\cup X$ for some $(i,j)\in V(\gp)$ and $X\subseteq V(\gp)$ such that $(i,j)\nsim(0)_0$ and $\gp[\un{X}{(i,j)}]$ is a complete graph. Since $(i,j)\in V(\gp)$, we have $i\in[0,m-1]$.
		\begin{itemize}
			\item Suppose $i=0$. Then $(i,j)=(j)_0\nsim(0)_0$ implies $j\in\di{[n-2]}{1}$. Since $\gp[\un{X}{(j)_0}]$ is a complete graph, $X$ satisfies one of the following conditions: (i) $X=\emptyset$, (ii) $X=\{(j-1)_0\}$, (iii) $X=\{(n-1)_0\}$ when $j=n-2$, or (iv) $X\subseteq\{(i',j)\, |\, i'\in[m-1]\}$ with $X\ne\emptyset$.
			
			Using Proposition \ref{proposition:Km_Cn not in C_n} and following the proof of Theorem \ref{theorem:K_m[]P_n}, cases (i), (ii), and (iv) contradict $\s\in\C_n$. Hence, the only possibility is $X=\{(n-1)_0\}$ when $j=n-2$. It follows that $\s^c=\{(0)_0,(n-2)_0,(n-1)_0\}$, and thus $\s\in A$.
			
			\item Suppose $i\in[m-1]$. In this case, $(i,j)\nsim(0)_0$ implies $j\in[n-1]$. Since $\gp[\un{X}{(i,j)}]$ is a complete graph, $X$ satisfies one of the following conditions: (i) $X=\emptyset$, (ii) $X=\{(i,j-1)\}$, (iii) $X=\{(i,0)\}$ when $j=n-1$, or (iv) $X\subseteq\{(i',j)\, |\, i'\in\di{[0,m-1]}{i}\}$ with $X\ne\emptyset$.
			
			Again, following the proof of Theorem \ref{theorem:K_m[]P_n} and using Proposition \ref{proposition:Km_Cn not in C_n}, cases (i) and (iv) contradict $\s\in\C_n$. Hence, either $X=\{(i,j-1)\}$, or $X=\{(i,0)\}$ when $j=n-1$. 
			This means that either $\s^c=\{(0)_0,(i,j-1),(i,j)\}\text{ for some }i\in[m-1]\text{ and }j\in[n-1]$, or $\s^c=\{(0)_0,(i,0),(i,n-1)\}$ for some $i\in[m-1]$. Therefore $\s\in B\sqcup C$.
		\end{itemize}
		We conclude that $\s\in A\sqcup B\sqcup C$.
		
		Conversely, let $\s\in A\sqcup B\sqcup C$.
		\begin{itemize}
			\item Suppose $\s\in A$. Then $\s^c=\{(0)_0,(n-2)_0,(n-1)_0\}$. Further, $(0)_0\nsim(n-2)_0$ and $(n-2)_0\sim(n-1)_0$ implies $\s\in\Dt{2}{\gp}$ and $\un{\s}{0}\notin\Dt{2}{\gp}$. Therefore $\s\in\C_1$ by Proposition \ref{proposition:first matching} \ref{sigma (not) in C iff}.
			
			Let $1\le r\le n-3$. Then $(r)_0\in\s$. Since $1<(n-1)-r<n-1$, $(r)_0\nsim(n-1)_0$. Hence $\un{(\di{\s}{(r)_0})}{(0)_0}\in\Dt{2}{\gp}$, and thus $\di{\s}{(r)_0}\notin\C_1$ by Proposition \ref{proposition:first matching} \ref{sigma (not) in C iff}. From Remark \ref{remark:C_i} \ref{in C_j+1}, $\s\in\C_{n-2}$.
			
			Now, since $(0)_0\sim(n-1)_0$, we have $\un{\s}{(n-2)_0}\notin\Dt{2}{\gp}$. Then $\C_{n-2}\subseteq\Dt{2}{\gp}$ implies $\un{\s}{(n-2)_0}\notin\C_{n-2}$. By Remark \ref{remark:C_i} \ref{in C_i to C_i+1}, $\s\in\C_{n-1}$.
			
			We have $\un{\s}{(n-1)_0}=\{(0)_0,(n-2)_0\}$. Clearly, $\un{\s}{(0)_0,(n-1)_0}\notin\Dt{2}{\gp}$. By Proposition \ref{proposition:first matching} \ref{sigma (not) in C iff}, $\un{\s}{(n-1)_0}\in\C_1$. Applying Proposition \ref{proposition:Km_Cn not in C_n} \ref{Km_Cn:0 nsim 0} to $\un{\s}{(n-1)_0}$ (for $j=n-2$), we get $\un{\s}{(n-1)_0}\notin\C_{n-2}$. Since $\C_{n-1}\subseteq\C_{n-2}$, $\un{\s}{(n-1)_0}\notin\C_{n-1}$. Thus $\s\in\C_{n}$ by Remark \ref{remark:C_i} \ref{in C_i to C_i+1}.
			
			\item Suppose $\s\in B$. Then $\s^c=\{(0)_0,(i,0),(i,n-1)\}$ for some $i\in[m-1]$. Since $(0)_0\nsim(i,n-1)$ and $(i,0)\sim(i,n-1)$, it follows that $\s\in\Dt{2}{\gp}$ and $\un{\s}{0}\notin\Dt{2}{\gp}$. By Proposition \ref{proposition:first matching} \ref{sigma (not) in C iff}, $\s\in\C_1$.
			
			Let $1\le r\le n-1$. Then $(r)_0\in\s$. Since $i\ne0$ and $r\ne0$, $(r)_0\nsim(i,0)$. Hence $\un{(\di{\s}{(r)_0})}{(0)_0}\in\Dt{2}{\gp}$. By Proposition \ref{proposition:first matching} \ref{sigma (not) in C iff}, $\di{\s}{(r)_0}\notin\C_1$. Therefore $\s\in\C_{n}$ by Remark \ref{remark:C_i} \ref{in C_j+1}.
			
			\item Suppose $\s\in C$. Then $\s^c=\{(0)_0,(i,j-1),(i,j)\}$ for some $i\in[m-1],\,j\in[n-1]$. Using the same argument as in the converse part of the proof of Theorem \ref{theorem:K_m[]P_n}, we get $\s\in\C_n$.
		\end{itemize}
		
		We conclude that $\C_n=A\sqcup B\sqcup C$. Note that $|A|=1$, $|B|=m-1$, and $|C|=(m-1)(n-1)$. Thus, $|A|+|B|+|C| = 1+(m-1)+(m-1)(n-1)=n(m-1)+1$. Further, each $\s\in A\sqcup B\sqcup C$ has dimension $mn-4$. Hence the result follows by Theorem \ref{theorem:acyclic} and Corollary \ref{corollary:critical}. 
	\end{proof} 
	
	\section{Conclusion and Future Directions}\label{section:future_direction}
	This article focuses on total $2$-cut complexes of the $p$-th powers of cycle graphs $C_n^p$, the Cartesian product of complete graphs and path graphs $K_m \square P_n$, and the Cartesian product of complete graphs and cycle graphs $K_m \square C_n$.   
	
	We have determined the homotopy type of $\Dt{2}{C_n^p}$ for $n=2p+2$ and $n\ge3p+1$, thus proving Conjecture \ref{conjecture:shen}. For $2p+3\le n\le 3p$, we investigated $\Dt{2}{C_n^p}$ using SageMath by computing their homology groups (with coefficients in $\mathbb{Z}$). 
	Table \ref{tab:2-cut_powers_of_cycle_graphs} summarizes the homology data for $\Dt{2}{\c}$ for small values of $n$ and $p$. Blank entries in the table indicate that the corresponding complex $\Dt{2}{\c}$ is void. An entry of the form \homo{i}{\bt} denotes that the $i^{\text{th}}$ homology group is $\mathbb{Z}^{\bt}$.  Entries above and below the colored region correspond to $n = 2p+2$ and $n \geq 3p+1$ respectively, which we have already proved in Theorem \ref{theorem:Intro_powers_of_cycle}. 
	Entries in the colored region correspond to $2p+3\le n\le 3p$. 
	
	\begin{table}[h!]\tiny
		\centering
		\begin{tabular}{|p{.23cm}|*{6}{c|}}\hline 
			{\diagbox[linewidth=0.2pt,width=.6cm,height=.5cm]{$n$}{$p$}} & $3$ & $4$ & $5$ & $6$ & $7$ & $8$ \\\hline
			
			$8$ & \homo{2}{} & & & & & \\\hline
			
			$9$ & \cfill{\homo{4}{2}} & & & & & \\\hline
			
			$10$ & \homo{6}{} & \homo{3}{} & & & & \\\hline
			
			$11$ & \homo{7}{} & \cfill{\homo{5}{}} & & & & \\\hline
			
			$12$ & \homo{8}{} & \cfill{\homo{7}{3}} & \homo{4}{} & & & \\\hline
			
			$13$ & \homo{9}{} & \homo{9}{} & \cfill{\homo{7}{}} & & & \\\hline
			
			$14$ & \homo{10}{} & \homo{10}{} & \cfill{\homo{8}{}} & \homo{5}{} & & \\\hline
			
			$15$ & \homo{11}{} & \homo{11}{} & \cfill{\homo{10}{4}} & \cfill{\homo{8}{2}} & & \\\hline
			
			$16$ & \homo{12}{} & \homo{12}{} & \homo{12}{} & \cfill{\homo{10}{}} & \homo{6}{} & \\\hline
			
			$17$ & \homo{13}{} & \homo{13}{} & \homo{13}{} & \cfill{\homo{11}{}} & \cfill{\homo{9}{}} & \\\hline
			
			$18$ & \homo{14}{} & \homo{14}{} & \homo{14}{} & \cfill{\homo{13}{5}} & \cfill{\homo{12}{}} & \homo{7}{} \\\hline
			
			$19$ & \homo{15}{} & \homo{15}{} & \homo{15}{} & \homo{15}{} & \cfill{\homo{13}{}} & \cfill{\homo{11}{}} \\\hline
			
			$20$ & \homo{16}{} & \homo{16}{} & \homo{16}{} & \homo{16}{} & \cfill{\homo{14}{}} & \cfill{\homo{13}{3}} \\\hline
			
			$21$ & \homo{17}{} & \homo{17}{} & \homo{17}{} & \homo{17}{} & \cfill{\homo{16}{6}} & \cfill{\homo{15}{}} \\\hline
			
			$22$ & \homo{18}{} & \homo{18}{} & \homo{18}{} & \homo{18}{} &  \homo{18}{} & \cfill{\homo{16}{}} \\\hline
			
			$23$ & \homo{19}{} & \homo{19}{} & \homo{19}{} & \homo{19}{} & \homo{19}{} & \cfill{\homo{17}{}} \\\hline
			
			$24$ & \homo{20}{} & \homo{20}{} & \homo{20}{} & \homo{20}{} &  \homo{20}{} & \cfill{\homo{19}{7}} \\\hline
			
			$25$ & \homo{21}{} & \homo{21}{} & \homo{21}{} & \homo{21}{} &  \homo{21}{} & \homo{21}{} \\\hline
		\end{tabular}
		\caption{\small{Non-zero homology $\Dt{2}{\c}$}}
		\label{tab:2-cut_powers_of_cycle_graphs}
	\end{table}
	
	Based on the data in Table \ref{tab:2-cut_powers_of_cycle_graphs}, we propose the following conjecture about the homotopy type of $\Dt{2}{\c}$ for $2p+3\le n\le 3p$.
	
	\begin{conjecture}\label{conjecture:2-cut power cycle}
		Let $2p+3\le n\le 3p$, and let $r$ be a positive integer such that $2p+\dfrac{p}{r+1}<n\le2p+\dfrac{p}{r}$.
		\begin{enumerate}[label=(\roman*)]
			\item If $n=2p+\dfrac{p}{r}$, then $\Dt{2}{\c}\simeq \bigvee_{n-2p-1}\mathbb{S}^{n-2r-3}$.
			
			\item If $2p+\dfrac{p}{r+1}<n<2p+\dfrac{p}{r}$, then 
			$\Dt{2}{\c}\simeq\mathbb{S}^{n-2r-4}.$           
		\end{enumerate}
	\end{conjecture}     
	
	Observe that for fixed $p$ and each $n$ with $2p+3\le n\le 3p$, there exists a unique positive integer $r$ such that $2p+\dfrac{p}{r+1}<n\le2p+\dfrac{p}{r}$. Hence, Conjecture \ref{conjecture:2-cut power cycle} addresses all $n$ satisfying $2p+3\le n\le 3p$. 
	
	The powers of cycle graphs $\c$ are circulant graphs $C_n(\{1,2,\ldots,p\})$ and hence Cayley graphs of $\bb{Z}_n$. This raises the following natural question.
	\begin{question}
		What are the homotopy types of  total cut complexes and cut complexes of general circulant graphs? More broadly, what about arbitrary Cayley graphs?
	\end{question}
	
	We have also determined the homotopy types of $\Dt{2}{K_m\square P_n}=\D{2}{K_m\square P_n}$ and $\Dt{2}{K_m\square C_n}=\D{2}{K_m\square C_n}$. Using SageMath, we further investigated the total $3$-cut and $3$-cut complexes of these Cartesian products. 
	We computed their homology groups (with coefficients in $\mathbb{Z}$) for some small values of $m$ and $n$, which are presented in  Tables \ref{tab:total 3-cut Km[]Pn}, \ref{tab:3-cut Km[]Pn}, \ref{tab:total 3-cut Km[]Cn} and \ref{tab:3-cut Km[]Cn}. 
	The notation \homo{i}{\bt} denotes that the $i^{\text{th}}$ homology group is $\mathbb{Z}^{\bt}$, and $\ast$ indicates that the corresponding entry is not known.
	The patterns observed in these tables lead us to Conjectures \ref{conjecture:total 3-cut K_m box P_n} to \ref{conjecture 3-cut K_m box C_n}. 
	
	\begin{table}[h!]\tiny
		\centering
		\begin{tabular}{|p{.23cm}|*{5}{c|}}\hline
			\diagbox[linewidth=0.2pt,width=.6cm,height=.5cm]{$m$}{$n$}
			& $3$ & $4$ & $5$ & $6$ & $7$ \\\hline
			
			$2$ & \homo{0}{} & \homo{2}{3} & \homo{4}{6} & \homo{6}{10} & \homo{8}{15} \\\hline
			
			$3$ & \homo{3}{4} & \homo{6}{12} & \homo{9}{24} & \homo{12}{40} & \homo{15}{60} \\\hline
			
			$4$ & \homo{6}{9} & \homo{10}{27} & \homo{14}{54} & $\ast$ & $\ast$ \\\hline
			
			$5$ & \homo{9}{16} & \homo{14}{48} & $\ast$ & $\ast$ & $\ast$ \\\hline
			
			$6$ & \homo{12}{25} & \homo{18}{75} & $\ast$ & $\ast$ & $\ast$ \\\hline
			
			$7$ & \homo{15}{36} & $\ast$ & $\ast$ & $\ast$ & $\ast$ \\\hline
		\end{tabular}
		\caption{\small Non-zero homology of $\Dt{3}{K_m \square P_n}$}
		\label{tab:total 3-cut Km[]Pn}
	\end{table}
	
	\begin{conjecture}\label{conjecture:total 3-cut K_m box P_n}
		For $m\ge2$ and $n\geq 3$, $\tilde{H}_i(\Dt{3}{K_m \square P_n}) =\begin{cases}
			\bb{Z}^{\frac{1}{2}(n^2-3n+2)(m-1)^2}&\text{if }i=mn-6,\\
			0&\text{otherwise}.
		\end{cases}$
	\end{conjecture} 
	
	\begin{table}[h!]\tiny
		\centering
		\begin{tabular}{|p{.23cm}|*{4}{c|}}\hline
			\diagbox[linewidth=0.2pt,width=.6cm,height=.5cm]{$m$}{$n$} & $3$ & $4$ & $5$ & $6$ \\\hline
			
			\multirowcell{2}{$3$} & \homo{4}{2} & \homo{7}{3} & \homo{10}{4} & \homo{13}{5} \\
			
			& \homo{5}{6} & \homo{8}{21} & \homo{11}{45} & \homo{14}{78} \\\hline
			
			\multirowcell{2}{$4$} & \homo{7}{6} & \homo{11}{9} & \homo{15}{12} & \multirowcell{2}{$\ast$} \\
			
			& \homo{8}{12} & \homo{12}{40} & \homo{16}{84} & \\\hline
			
			\multirowcell{2}{$5$} & \homo{10}{12} & \homo{15}{18} & \multirowcell{2}{$\ast$} & \multirowcell{2}{$\ast$} \\
			
			& \homo{11}{20} & \homo{16}{65} & & \\\hline
			
			\multirowcell{2}{$6$} & \homo{13}{20} & \multirowcell{2}{$\ast$} & \multirowcell{2}{$\ast$} & \multirowcell{2}{$\ast$} \\
			
			& \homo{14}{30} & & & \\\hline
		\end{tabular}
		\caption{\small Non-zero homology of $\D{3}{K_m \square P_n}$}
		\label{tab:3-cut Km[]Pn}
	\end{table}
	
	\begin{conjecture}\label{conjecture:3-cut K_m box P_n}
		For $m, n\geq 3$,
		$\tilde{H}_i(\D{3}{K_m \square P_n}) = 
		\begin{cases}
			\mathbb{Z}^{\frac{1}{2}(m^2-3m+2)(n-1)} & \text{if } i = mn-5, \\
			\mathbb{Z}^{\frac{m}{2}(mn-m-2)(n-2)} & \text{if } i = mn-4, \\
			0 & \text{otherwise}.\\
		\end{cases}$
	\end{conjecture}
	
	\begin{table}[h!]\tiny
		\centering
		\begin{tabular}{|p{.23cm}|*{8}{c|}}\hline
			\diagbox[linewidth=0.2pt,width=.6cm,height=.5cm]{$m$}{$n$}
			& $3$ & $4$ & $5$ & $6$ & $7$ & $8$ & $9$ & $10$ \\\hline
			
			$2$ & void complex & \homo{2}{9} & \homo{4}{14} & \homo{6}{22} & \homo{8}{29} & \homo{10}{37} & \homo{12}{46} & \homo{14}{56} \\\hline
			
			$3$ & \homo{3}{} & \homo{6}{30} & \homo{9}{48} & \homo{12}{73} & \homo{15}{99} & $\ast$ & $\ast$ & $\ast$ \\\hline
			
			$4$ & \homo{6}{3} & \homo{10}{63} & \homo{14}{102} & $\ast$ & $\ast$ & $\ast$ & $\ast$ & $\ast$ \\\hline
			
			$5$ & \homo{9}{6} & \homo{14}{108} & $\ast$ & $\ast$ & $\ast$ & $\ast$ & $\ast$ & $\ast$ \\\hline
			
			$6$ & \homo{12}{10} & \homo{18}{165} & $\ast$ & $\ast$ & $\ast$ & $\ast$ & $\ast$ & $\ast$ \\\hline
			
			$7$ & \homo{15}{15} & $\ast$ & $\ast$ & $\ast$ & $\ast$ & $\ast$ & $\ast$ & $\ast$ \\\hline
		\end{tabular}
		\caption{\small Non-zero homology of $\Dt{3}{K_m \square C_n}$}
		\label{tab:total 3-cut Km[]Cn}
	\end{table}
	
	\begin{conjecture}\label{conjecture:total 3-cut K_m box C_n}
		\begin{enumerate}[label=(\roman*)]
			\item For $m\ge3$ and $n=3$, $\tilde{H}_i(\Dt{3}{K_m\square C_n})=\begin{cases}
				\bb{Z}^{\frac{1}{2}(m^2-3m+2)}&\text{if }i=mn-6,\\
				0&\text{otherwise}.
			\end{cases}$
			
			\item For $m\ge2$ and $n\in\{4,5\}$, $\tilde{H}_i(\Dt{3}{K_m\square C_n})=\begin{cases}
				\bb{Z}^{\frac{m-1}{2}((m-1)n^2-(m-3)n-2)}&\text{if }i=mn-6,\\
				0&\text{otherwise}.
			\end{cases}$
			
			\item For $m\ge2$ and $n\ge6$, $\tilde{H}_i(\Dt{3}{K_m\square C_n})=\begin{cases}
				\bb{Z}^{\frac{1}{2}((m-1)^2n^2-(m^2-4m+3)n+2)}&\text{if }i=mn-6,\\
				0&\text{otherwise}.
			\end{cases}$
		\end{enumerate}
	\end{conjecture}
	
	\begin{table}[h!]\tiny  
		\centering
		\begin{tabular}{|p{.23cm}|*{5}{c|}}\hline
			\diagbox[linewidth=0.2pt,width=.6cm,height=.5cm]{$m$}{$n$} & $4$ & $5$ & $6$ & $7$ & $8$ \\\hline
			
			\multirowcell{2}{$2$} & \homo{3}{} & \multirowcell{2}{\homo{6}{11}} & \multirowcell{2}{\homo{8}{25}} & \multirowcell{2}{\homo{10}{43}} & \multirowcell{2}{\homo{12}{65}} \\
			
			& \homo{4}{4} & & & & \\\hline
			
			\multirowcell{2}{$3$} & \homo{7}{6} & \homo{10}{5} & \homo{13}{6} & \multirowcell{2}{$\ast$} & \multirowcell{2}{$\ast$} \\
			
			& \homo{8}{12} & \homo{11}{31} & \homo{14}{64} & & \\\hline
			
			\multirowcell{2}{$4$} & \homo{11}{15} & \homo{15}{15} & $\ast$ & $\ast$ & $\ast$ \\
			
			& \homo{12}{24} & \homo{16}{61} & $\ast$ & $\ast$ & $\ast$ \\\hline
			
			\multirowcell{2}{$5$} & \homo{15}{28} & $\ast$ & $\ast$ & $\ast$ & $\ast$ \\
			
			& \homo{16}{40} & $\ast$ & $\ast$ & $\ast$ & $\ast$ \\\hline
			
		\end{tabular}\caption{\small Non-zero homology of $\D{3}{K_m \square C_n}$}
		\label{tab:3-cut Km[]Cn}
	\end{table}
	
	\begin{conjecture}\label{conjecture 3-cut K_m box C_n}
		\begin{enumerate}[label=(\roman*)]
			\item For $m\ge2$ and $n=4$, $\tilde{H}_i(\D{3}{K_m \square C_n}) =\begin{cases}
				\mathbb{Z}^{2m^2-5m+3} & \text{if}\  i = mn-5, \\
				\mathbb{Z}^{2m(m-1)} & \text{if}\ i = mn-4, \\
				0 & \text{otherwise}.\\
			\end{cases}$
			
			\item For $m=2$ and $n\ge5$, $\tilde{H}_i(\D{3}{K_m \square C_n}) =\begin{cases}
				\mathbb{Z}^{2n^2-8n+1} & \text{if } i = mn-4, \\
				0 & \text{otherwise}.\\
			\end{cases}$
		\end{enumerate}
	\end{conjecture}
	
	\section*{Acknowledgement}
	The first author is supported by HTRA fellowship by IIT Mandi, India. The second author is supported by the seed grant project IITM/SG/SMS/95 by IIT Mandi, India.
	
	\bibliographystyle{abbrv}
	\bibliography{ref}
	
	\addcontentsline{toc}{section}{References}
	
\end{document}